\documentclass[12pt,reqno]{amsart}
 \usepackage[utf8]{inputenc}
 \usepackage[T1]{fontenc}
 \usepackage[usenames, dvipsnames]{color}
 \usepackage{ulem}
 
 \usepackage{dsfont, amsfonts, amsmath, amssymb,amscd, stmaryrd, latexsym, amsthm, dsfont}
 \usepackage[frenchb,english]{babel}
 \usepackage{enumerate}
 \usepackage{longtable}
 \usepackage{geometry}
 \usepackage{float}
 \usepackage{tikz}
 \usetikzlibrary{shapes,arrows}
 \usepackage{float}
 \usepackage{tikz}
 \usepackage{xypic}
 \usetikzlibrary{shapes,arrows}
 \usepackage{pifont}

 \newtheorem{theorem}{Theorem}[section]
 \newtheorem{lemma}[theorem]{Lemma}
 \newtheorem{proposition}[theorem]{Proposition}
 \newtheorem{corollary}[theorem]{Corollary}
 
 \theoremstyle{remark}
 \newtheorem{remark}[theorem]{\bf Remark}

 \usepackage[pagebackref]{hyperref}
 \renewcommand*{\backref}[1]{}\renewcommand*{\backrefalt}[4]{\ifcase #1 (\tt not cited)\or (\tt cited on page~#2)\else (\tt cited on pages~#2)\fi}
 
 \usepackage{hyperref}
 \hypersetup{
 	colorlinks=true,
 	urlcolor=blue,
 	citecolor=blue}
 
 \def\NN{\mathds{N}}
 
 \def\QQ{\mathbb{Q}}
 
 \def\ZZ{\mathbb{Z}}
 \def\kk{\mathds{k}}
 \def\po{p_1}
 \def\pt{p_2}
 \def\cl{\mathbf{C}l}

 \usepackage{mathtools}
 \usepackage{mleftright}

 \def\NN{\mathbb{N}}

 \def\QQ{\mathbb{Q}}
 
 \def\ZZ{\mathbb{Z}}

 \def\kk{\mathds{k}}
 \def\KK{\mathbb{K}}
 
 \def\LL{\mathbb{L}}

\begin{document}
 	\selectlanguage{english}
 	\title[On the unit group   of some multiquadratic   fields]{On the unit group   of some multiquadratic   fields}

 	\author[M. M. Chems-Eddin]{Mohamed Mahmoud CHEMS-EDDIN}
 	\address{Mohamed Mahmoud CHEMS-EDDIN: Department of Mathematics, Faculty of Sciences Dhar El Mahraz, Sidi Mohamed Ben Abdellah University, Fez,  Morocco}
 	\email{2m.chemseddin@gmail.com}
 	
 	\author[H. El Mamry]{Hamza El Mamry}
 	\address{Hamza EL MAMRY: Departement of Mathematics, Faculty of Sciences Dhar El Mahraz,  Sidi Mohamed Ben Abdellah University, Fez, Morocco}
 	\email{Hamza.elmamry@usmba.ac.ma}

 	\subjclass[2010]{11R04, 11R27, 11R29, 11R37.}
 	\keywords{Real multiquadratic   fields,  unit group, $2$-class group, $2$-class  number.}
 	
 	\maketitle
 	\begin{abstract} 
 		In this paper, we calculate the unit groups  and the $2$-class numbers of the   fields $ \KK= \mathbb{Q}(\sqrt{2}, \sqrt{\po}, \sqrt{\pt})$ and   $ \LL= \mathbb{Q}( \sqrt{-1},\sqrt{2}, \sqrt{\po}, \sqrt{\pt})$, where $\po$ and $\pt$ are two prime numbers satisfying $\po\equiv  \pt \equiv 1 \pmod{4}$.
 	\end{abstract}
 	
 	\section*{Introduction}

 	Let  $k$ be a  number field of degree $n$  (i.e., $[k: \QQ] = n$).  Denote by $E_k$  the unit group of $k$ that is  the group of the invertible elements of the ring $\mathcal{O}_k$ of algebraic integers of the number field  $k$.   By the well known Dirichlet's unit theorem, if $n=r+2s$, where $r$ is the number of real embeddings and $s$ the number of conjugate pairs of complex embeddings of $k$, then there exist $t=r+s-1$ units of $\mathcal{O}_k$ that generate $E_k$ (modulo the roots of unity), and these $t$ units are called the {\it  fundamental system of units of $k$}. Therefore, it is well known that
 	$$E_k\simeq \mu(k)\times \mathbb{Z}^{r+s-1},$$
 	where   $\mu(k)$  is the group of roots of unity contained in $k$.
 	
 	One major  problem in algebraic number theory (more precisely, in the theory of units of number fields which is related to almost all areas of algebraic number theory) is the computation of the
 	fundamental system of units. For  quadratic fields, the problem
 	is easily solved. An early study of unit groups of multiquadratic fields was established by Varmon \cite{Varmon}.    For quartic bicyclic fields, Kubota \cite{Ku-56} gave
 	a method for finding a fundamental system of units.
 	Wada \cite{wada} generalized Kubota's method, creating an algorithm for computing
 	fundamental units in any given multiquadratic field. However, in general, it is not easy to  compute the unit group of a number field  especially for number fields of degree more than $4$. Many authors tried  to characterize their Hasse unit index. For instance, Hirabayashi      and  Yoshino \cite{HY,HY2} gave criteria to determine the Hasse unit index of some multiquadratic number fields of degree $4$ and $8$. 
 	In \cite{H4,H3},   Hirabayashi   gave a necessary and sufficient condition for  a multiquadratic  field of degree 16 to have a unit index equal  to $2$.  	For more works in the same direction, we refer the reader to \cite{azizunints99,azizitalbi,benjlemschi,Bulant}.
 	
 	Actually, in  literature  there are only few examples of the explicit computation of the unit group of a given number field $k$ of degree $\geq 8$ (cf.   \cite{azizireg,ChemsUnits9,CAZ,chemszekhniniaziziUnits1}).

 	Recently, the  authors of \cite{Chemsarith,CEH2024}  computed the unit group of the fields $\QQ(\sqrt 2, \sqrt{p}, \sqrt{q} )$, where $p\equiv 1$, $3  \pmod{4}$ and $q\equiv3\pmod 4$ are two prime  numbers.

 	As  a continuation, in present work,   we focus on the computation of the unit group of the real triquadratic fields of the form 
 	$\KK=\QQ(\sqrt 2, \sqrt{p_1}, \sqrt{p_2} )$, where $p_1$ and $p_2$ are two primes such that  $p_1\equiv p_2\equiv 1 \pmod{4}$. We note that this case has particular interest. In fact, according to \cite[Proposition 1]{H4} the Hasse unit index  of the field $ \LL= \mathbb{Q}( \sqrt{-1},\sqrt{2}, \sqrt{\po}, \sqrt{\pt})$ equals $1$ (i.e. $(E_{\LL}:\mu(\LL)E_{\KK})=1$), so 
 	\begin{eqnarray}\label{frem}
 		E_{\LL}= \langle \zeta_8\rangle E_{\KK}
 	\end{eqnarray}
 	So our Theorems \ref{MT1A}, \ref{MT3}  and  \ref{MT4} (cf. pages \pageref{MT1A},  \pageref{MT3} and  \pageref{MT4} respectively) give  also the unit group of $\LL$.
 	Notice that the motivation behind the computation of the unit groups of these fields is the fact that $\KK$ (resp. $\LL$) is the first layer of the cyclotomic  $\ZZ_2$-extension of the biquadratic field $\QQ(\sqrt{p_1}, \sqrt{p_2} )$ (resp. $\QQ(\sqrt{p_1}, \sqrt{p_2},  \sqrt{-1})$),
 	 see \cite{ChemsUnits9,CAZ} for examples of applications in this direction. Furthermore, computing the unit group of the fields  $\KK$  is also the first step to find the unit group of all fields of the form $\KK(\sqrt{-\ell})$, where $\ell\textgreater 1$ is a positive square-free integer (cf. \cite{azizunints99}). We note that the unit group of these fields are useful for the study of the Hilbert $2$-class field tower of the subfields $\KK(\sqrt{-\ell})$ 
 	 (e.g., \cite{azizireg,chemszekhniniaziziUnits1}). 
 	
 	The plan of this paper is the following. In the next section		(i.e.  Section	\ref{sec2prep} (p. \pageref{sec2prep})), we collect some preliminaries and preparations that will be helpful to build our proofs later. In Section \ref{casestable} (p. \pageref{casestable}), we compute the unit group of $\KK$ when $p_1\equiv p_2\equiv 5 \pmod{8}$.
 	Section  \ref{section2} (p. \pageref{section2}) is dedicated to the computation of the unit group of $\KK$, when 
 	$p_1\equiv1 \pmod{8}$,   $p_2 \equiv 1 \pmod{4}$ and $(N(\varepsilon_{p_1p_2})   ,N(\varepsilon_{2p_1p_2}))\not=(-1,-1)$.  
 	In Section \ref{section23} (p. \pageref{section23}), we compute the unit group of $\KK$ for the remaining cases. Finally in Section \ref{lastsec} (p. \pageref{lastsec}), we give  some concluding    results and remarks.
 	
 	Throughout this paper we use  the following notations. Let $k$ be a number field.
 	\begin{enumerate}[$\star$]
 		\item $E_k$: the unit group of $k$,
 		\item $\mathbf{C}l_2(k)$: the $2$-class group of $k$,
 		\item $q(k)=(E_{k}: \prod_{i}E_{k_i})$ is the unit index of $k$, if $k$ is multiquadratic, where   $k_i$ are  the  quadratic subfields	of $k$,
 		\item $h(k)$: the  class number of  $k$,
 		\item $h_2(k)$: the $2$-class number of $k$,
 	\item	$k^{(n)}$: the $n$th Hilbert $2$-class field of $k$,
 		\item $h_2(d)$: the $2$-class number of a quadratic field $\mathbb{Q}(\sqrt{d})$,
 		\item $\varepsilon_d$: the fundamental unit of a real quadratic field $\mathbb{Q}(\sqrt{d})$,
 		\item $N(\varepsilon_d)$: the norm of   $\varepsilon_d$ in the extension $\mathbb{Q}(\sqrt{d})/\mathbb{Q}$,
 		\item $p_1$, $p_2$: two  different prime numbers such that 	$\po\equiv  \pt \equiv 1 \pmod 4$,
 		\item $\tau_i$: defined in Page \pageref{fsu preparations},
 		\item$k_i$: defined in Page \pageref{fsu preparations},
 		\item $\left(\frac{\alpha,d}{\mathfrak{p}}\right)$:  the quadratic norm residue symbol over   $k$,
 		\item $\left(\frac{\cdot}{\cdot}\right)$: The Legendre symbol,
 		\item $\left(\frac{\cdot}{\cdot}\right)_4$: The biquadratic symbol,
 		\item $n_1$ (resp. $n_2$): the norm of the unit $  \varepsilon_{2p_1}$ (resp. $ \varepsilon_{2p_2}$),
 		\item$n_3$ (resp. $n_4$): the norm of the unit $  \varepsilon_{p_1p_2}$ (resp. $\varepsilon_{2p_1p_2}$).
 	\end{enumerate}
 	\section{\bf First Preliminaries } \label{sec2prep}	
 	Let us start this section by  recalling the method given in    \cite{wada}, that describes a fundamental system  of units of a real  multiquadratic field $K_0$. Let  $\tau_1$ and 
 	$\tau_2$ be two distinct elements of order $2$ of the Galois group of $K_0/\mathbb{Q}$. Let $K_1$, $K_2$ and $K_3$ be the three subextensions of $K_0$ invariant by  $\tau_1$,
 	$\tau_2$ and $\tau_1\tau_2$, respectively. Let $\varepsilon$ denote a unit of $K_0$. Then \label{algo wada}
 	$$\varepsilon^2=\varepsilon\varepsilon^{\tau_1}  \varepsilon\varepsilon^{\tau_2}(\varepsilon^{\tau_1}\varepsilon^{\tau_2})^{-1},$$
 	and we have, $\varepsilon\varepsilon^{\tau_1}\in E_{K_1}$, $\varepsilon\varepsilon^{\tau_2}\in E_{K_2}$  and $\varepsilon^{\tau_1}\varepsilon^{\tau_2}\in E_{K_3}$.
 	It follows that the unit group of $K_0$  
 	is generated by the elements of  $E_{K_1}$, $E_{K_2}$ and $E_{K_3}$, and the square roots of elements of   $E_{K_1}E_{K_2}E_{K_3}$ which are perfect squares in $K_0$.
 	
 	This method is very useful for computing a fundamental system of units of a real biquadratic number field, however, in the case of a real triquadratic 
 	number field the problem of the determination of the unit group becomes very difficult and demands some specific computations and eliminations, as we will see in the next sections.
 	We shall consider the field $\KK=\mathbb{Q}(\sqrt{2},\sqrt{\po},\sqrt{\pt})$, where $\po$ and $\pt$ are two distinct prime numbers. Thus, we have the following diagram:

 	\begin{figure}[H]
 		$$\xymatrix@R=0.8cm@C=0.3cm{
 			&\KK=\QQ( \sqrt 2, \sqrt{\po}, \sqrt{\pt})\ar@{<-}[d] \ar@{<-}[dr] \ar@{<-}[ld] \\
 			k_1=\QQ(\sqrt 2,\sqrt{\po})\ar@{<-}[dr]& k_2=\QQ(\sqrt 2, \sqrt{\pt}) \ar@{<-}[d]& k_3=\QQ(\sqrt 2, \sqrt{\po\pt})\ar@{<-}[ld]\\
 			&\QQ(\sqrt 2)}$$
 		\caption{Intermediate fields of $\KK/\QQ(\sqrt 2)$}\label{fig:I}
 	\end{figure}
 	
 	We  consider also the following fields   : $k_4=\QQ(\sqrt{\po},\sqrt{\pt}), k_5=\QQ(\sqrt{\pt},\sqrt{2\po})$,  $k_6=\QQ(\sqrt{\po},\sqrt{2\pt})$ and $k_7=\QQ(\sqrt{2\po}, \sqrt{2\pt})$.
 	\begin{remark}\label{unramified}
 		
 		We note that $\KK$ is unramified over   its subfields appearing in the below diagram:
 		\begin{center}
 			\begin{figure}[H]\label{tours}
 				
 				{
 					\begin{tikzpicture} [scale=1.2]
 						\node (k3)  at (-3, 0) {$k_3 $};
 						
 						\node (k6)  at (-1,  0) {$k_5 $};
 						\node (k5)  at (1,  0) {$k_6 $};
 						\node (k7)  at (3,  0) {$k_7 $};
 						\node (KK)  at (0,  1.5) {$\KK$};
 						\draw (k3) --(KK)  node[scale=0.4,  midway,  below right]{};
 						\draw (k5) --(KK)  node[scale=0.4,  midway,  below right]{};
 						\draw (k6) --(KK)  node[scale=0.4,  midway,  below right]{\Large };
 						\draw (k7) --(KK)  node[scale=0.4,  midway,  below right]{\Large };
 						
 				\end{tikzpicture}}
 				\caption{Unramified subfields of $\KK$.}\label{Fig 3}
 			\end{figure}
 		\end{center} 
 	\end{remark} 
 	Let $\tau_1$, $\tau_2$ and $\tau_3$ be the elements of  $ \mathrm{Gal}(\KK/\QQ)$ defined by
 	\begin{center}	\begin{tabular}{l l l }
 			$\tau_1(\sqrt{2})=-\sqrt{2}$, \qquad & $\tau_1(\sqrt{\po})=\sqrt{\po}$, \qquad & $\tau_1(\sqrt{\pt})=\sqrt{\pt},$\\
 			$\tau_2(\sqrt{2})=\sqrt{2}$, \qquad & $\tau_2(\sqrt{\po})=-\sqrt{\po}$, \qquad &  $\tau_2(\sqrt{\pt})=\sqrt{\pt},$\\
 			$\tau_3(\sqrt{2})=\sqrt{2}$, \qquad &$\tau_3(\sqrt{\po})=\sqrt{\po}$, \qquad & $\tau_3(\sqrt{\pt})=-\sqrt{\pt}.$
 		\end{tabular}
 	\end{center}
 	Note that  $\mathrm{Gal}(\KK/\QQ)=\langle \tau_1, \tau_2, \tau_3\rangle$
 	and the subfields  $k_1$, $k_2$ and $k_3$ are
 	fixed by  $\langle \tau_3\rangle$, $\langle\tau_2\rangle$ and $\langle\tau_2\tau_3\rangle$ respectively. Therefore,\label{fsu preparations} a fundamental system of units  of $\KK$ consists  of seven  units chosen from those of $k_1$, $k_2$ and $k_3$, and  from the square roots of the elements of $E_{k_1}E_{k_2}E_{k_3}$ which are squares in $\KK$.   
 	With these   notations, we have:
 	\begin{lemma}[\cite{azizitalbi}, Theorem 3]\label{fork1k2}
 		Let $p\equiv 1 \pmod 4$ be a prime number. We have:
 		\begin{enumerate}[$1)$]
 			\item If $N(\varepsilon_{2p})=-1$ then $\{\varepsilon_{2},\varepsilon_{p},\sqrt{\varepsilon_{2}\varepsilon_{p}\varepsilon_{2p}}\}$ is a F.S.U of $\QQ(\sqrt{2},\sqrt{p})$. 
 			\item If $N(\varepsilon_{2p})=1$ then $\{\varepsilon_{2},\varepsilon_{p},\sqrt{\varepsilon_{2p}}\}$ is a F.S.U of $\QQ(\sqrt{2},\sqrt{p})$.
 		\end{enumerate}
 		
 	\end{lemma}

 	\begin{lemma}[\cite{AzZektaous}, Theorem 2.2]\label{lmunit}
 		Let  $p_1\equiv p_2\equiv 1 \pmod 4$ be two prime numbers and $k_3=\QQ(\sqrt{2},\sqrt{\po\pt})$. Let  $\varepsilon_{2\po\pt}=x+y\sqrt{2\po\pt}$, where $x$ and $y$ are two integers. We have 
 		\begin{enumerate}[$1)$]
 			\item If $N(\varepsilon_{\po\pt})= N(\varepsilon_{2\po\pt})=-1$, then    
 			a F.S.U of $k_3$ is given by 
 			$$\{\varepsilon_{2}, \varepsilon_{\po\pt}, \sqrt{\varepsilon_{2}\varepsilon_{\po\pt}\varepsilon_{2\po\pt}}\} \text{ or }\{\varepsilon_{2}, \varepsilon_{\po\pt},  \varepsilon_{2\po\pt}\},$$ according to whether $\varepsilon_{2}\varepsilon_{\po\pt}\varepsilon_{2\po\pt}$ is a square in $k_3$ or not.

 			\item If $N(\varepsilon_{\po\pt})=-N(\varepsilon_{2\po\pt})=1$ then a F.S.U of $k_3$ is $\{\varepsilon_{2},\varepsilon_{\po\pt}, \varepsilon_{2\po\pt} \}$.
 			\item If $N(\varepsilon_{2\po\pt})=-N(\varepsilon_{\po\pt})=1$ then a F.S.U of $k_3$ is 
 			$$\{\varepsilon_{2}, \varepsilon_{\po\pt}, \sqrt{\varepsilon_{2\po\pt}}\}\text{ or }\{\varepsilon_{2}, \varepsilon_{\po\pt},  \varepsilon_{2\po\pt}\},$$ 
 			according to whether $x\pm1$ 
 			is a square in $\NN$ or not.

 			\item If $N(\varepsilon_{2\po\pt})=N(\varepsilon_{\po\pt})=1$ then a F.S.U of $k_3$ is 
 			$$\{\varepsilon_{2}, \varepsilon_{\po\pt}, \sqrt{\varepsilon_{2\po\pt}}\}\text{ or }\{\varepsilon_{2}, \varepsilon_{\po\pt},  \sqrt{\varepsilon_{\po\pt}\varepsilon_{2\po\pt}}\},$$ according to whether $x\pm1$ is a square in $\NN$ or not.
 		\end{enumerate}
 	\end{lemma}
 	
 	\begin{lemma}[\cite{AzZektaous}, Theorem 2.1]\label{biquadunutk6}
 		Let $\po\equiv\pt \equiv 1 \pmod 4$ be two primes and $k=\QQ(\sqrt{\po},\sqrt{2\pt})$. Put  $\varepsilon_{2\po\pt}=x+y\sqrt{2\po\pt}$, where $x$ and $y$ are two integers.
 		\begin{enumerate}[$1)$]
 			\item Assume that  $N(\varepsilon_{2\pt}) = N(\varepsilon_{2\po\pt}) = -1$. We have:
 			\begin{enumerate}[$a)$]
 				\item If $\varepsilon_{\po}\varepsilon_{2\pt}\varepsilon_{2\po\pt}$ is a square in $k$, then $\{\varepsilon_{\po}, \varepsilon_{2\pt}, \sqrt{\varepsilon_{\po}\varepsilon_{2\pt}\varepsilon_{2\po\pt}}\}$ is a F.S.U of $k$.
 				\item Else $\{\varepsilon_{\po}, \varepsilon_{2\pt}, \varepsilon_{2\po\pt}\}$ is a F.S.U of $k$.
 			\end{enumerate}
 			\item Assume that $N(\varepsilon_{2\pt}) = -N(\varepsilon_{2\po\pt}) = 1$, then  $\{\varepsilon_{\po}, \varepsilon_{2\pt}, \varepsilon_{2\po\pt}\}$ is a F.S.U of $k$.
 			\item Assume that $N(\varepsilon_{2\po\pt}) = -N(\varepsilon_{2\pt}) = 1$. We have:
 			\begin{enumerate}[$a)$]
 				\item If $2p_1(x \pm 1)$ is a square in $\mathbb{N}$, then $\{\varepsilon_{\po}, \varepsilon_{2\pt}, \sqrt{\varepsilon_{2\po\pt}}\}$ is a F.S.U of $k$.
 				\item Else $\{\varepsilon_{\po}, \varepsilon_{2\pt}, \varepsilon_{2\po\pt}\}$ is a F.S.U of $k$.
 			\end{enumerate}
 			\item Assume that $N(\varepsilon_{2\po\pt}) = N(\varepsilon_{2\pt}) = 1$. We have:
 			\begin{enumerate}[$a)$]
 				\item If $2p_1(x \pm 1)$ is a square in $\mathbb{N}$, then $\{\varepsilon_{\po}, \varepsilon_{2\pt}, \sqrt{\varepsilon_{2\po\pt}}\}$ is a F.S.U of $k$.
 				\item Else $\{\varepsilon_{\po}, \varepsilon_{2\pt}, \sqrt{\varepsilon_{2\pt}\varepsilon_{2\po\pt}}\}$ is a F.S.U of $k$.
 			\end{enumerate}
 		\end{enumerate}
 	\end{lemma}

 	Now we recall the following lemmas:

 	\begin{lemma}[\cite{Ku-50}]\label{wada's f.}
 		Let $K$ be a multiquadratic number field of degree $2^n$, $n\in\mathds{N}$,  and $k_i$ the $s=2^n-1$ quadratic subfields of $K$. Then
 		$$h_2(K)=\frac{1}{2^v}q(K)\prod_{i=1}^{s}h_2(k_i),$$
 		with $q(K):=( E_K: \prod_{i=1}^{s}E_{k_i})$ and  $$v=\left\{ \begin{array}{cl}
 			n(2^{n-1}-1), &\text{ if } K \text{ is real, }\\
 			(n-1)(2^{n-2}-1)+2^{n-1}-1, & \text{ if } K \text{ is imaginary.}
 		\end{array}\right.$$
 	\end{lemma}
 	\begin{lemma}\label{class numbers of quadratic field}
 		Let $p\equiv p'\equiv 1\pmod 4$ be two prime numbers. 
 		\begin{enumerate}[\rm 1.]
 			\item By \cite[Corollary 18.4]{connor88}, we have  $h_2(p)=h_2(2)=1$.
 			\item If $\genfrac(){}{0}{p}{p'} =-1$,   then   $h_2(p p')=2$, $($cf.   \cite[Corollary 19.8 ,  p. 147]{connor88}$)$.
 			\item If $ p\equiv 5\pmod 8$, then $h_2(2p)=2$ $($cf.   \cite[Corollary 19.8 ,  p. 147]{connor88}$)$.
 			\item  If two of the set $\left\{\genfrac(){}{0}{p}{ p'},\genfrac(){}{0}{2}{p}, \genfrac(){}{0}{2}{ p'}\right\}$ are equal to -1 , then $h_2(2p p')=4$  $($cf.   \cite[Proposition $A_2$,  p. 330]{kaplan76}$)$.
 			\item If $p\equiv 1\pmod 8$ and  $  \genfrac(){}{0}{2}{p}_4\not=  \genfrac(){}{0}{p}{2}_4$, then $h_2(2p )=2$  $($cf.   \cite[Theorem 2 (2)(a)]{kuvcera1995parity}$)$. 
 			\item If  $p\equiv 1\pmod 8$ and  $  \genfrac(){}{0}{2}{p}_4 =  \genfrac(){}{0}{p}{2}_4=-1$, then $h_2(2p )=4$  $($cf.   \cite[Theorem 2 (2)(b)]{kuvcera1995parity}$)$.
 		\end{enumerate}
 	\end{lemma}

 	\begin{lemma}[\cite{Qinred}, Lemma 2.4]\label{ambiguous class number formula} Let $K/k$ be a quadratic extension. If the class number of $k$ is odd, then the rank of the $2$-class group of $K$ is given by
 		$$r_2({Cl}(K))=t_{K/k}-1-e_{K/k},$$
 		where  $t_{K/k}$ is the number of  ramified primes (finite or infinite) in the extension  $K/k$ and $e_{K/k}$ is  defined by   $2^{e_{K/k}}=[E_{k}:E_{k} \cap N_{K/k}(K^*)]$.
 	\end{lemma}
 		For a morphism $\tau$ we use the expression $1+\tau$ to define the morphism $(1+\tau)(x):=x\cdot\tau(x)$

 	\begin{lemma}[\cite{CEH2024}, Lemma 2.1]\label{calcul}
 		
 		\textit{Let $p \equiv 1 \pmod{8}$ be a prime number. Put $\varepsilon_{2p} = \beta + \alpha\sqrt{2p}$ with $\beta, \alpha \in \mathbb{Z}, and\; assume \; N(\varepsilon_{2p}) = 1$. Then $\sqrt{\varepsilon_{2p}} = \frac{1}{\sqrt{2}}(\alpha_1 + \alpha_2\sqrt{2p})$, for some integers $\alpha_1, \alpha_2$ such that $\alpha = \alpha_1\alpha_2$. It follows that:}
 		
 		\begin{equation}
 			\begin{array}{|c|c|c|c|c|c|}
 				\hline
 				\sigma & 1 + \tau_2 & 1 + \tau_1\tau_2 & 1 + \tau_1\tau_3 & 1 + \tau_2\tau_3 & 1 + \tau_1 \\
 				\hline
 				\sqrt{\varepsilon_{2p}}^\sigma & (-1)^u & -\varepsilon_{2p} & (-1)^{u+1} & (-1)^u & (-1)^{u+1} \\
 				\hline
 			\end{array}
 			\tag{1}
 		\end{equation}
 		
 		\textit{for some $u$ in $\{0, 1\}$ such that $\frac{1}{2}(\alpha_1^2 - 2p\alpha_2^2) = (-1)^u$.}
 	\end{lemma}

 	\begin{lemma}[\cite{Az-00}, Lemma 5]\label{tettt}
 		Let $d>1$ be a square free interger and $\varepsilon_d=x+y\sqrt{d}$, where $x$ and $y$ are intergers or semi-integers. If $N(\varepsilon_d)=1$, then $2(x+1), 2(x-1), 2d(x+1)$ and $2d(x-1)$ are not squares in $\QQ$.    
 	\end{lemma}

 	\begin{remark}[\cite{AzZektaous}, Remark 2.3]\label{remarkzekh}
 		If $N(\varepsilon_{\po\pt})=N(\varepsilon_{2\po\pt})=-1$, then there exist four rational numbers  $\alpha, \beta,\gamma$ and $\delta$ such that :$$\sqrt{\varepsilon_{2}\varepsilon_{\po\pt}\varepsilon_{2\po\pt}}=\alpha +\beta\sqrt{2}+\gamma\sqrt{\po\pt}+\delta\sqrt{2\po\pt} $$
 		or 
 		$$\sqrt{\varepsilon_{2}\varepsilon_{\po\pt}\varepsilon_{2\po\pt}}=\alpha\sqrt{\po} +\beta\sqrt{\pt}+\gamma\sqrt{2\po}+\delta\sqrt{2\pt}. $$
 	\end{remark}
 	
 	Inspired by the proof of the previous remark, we get the following useful lemma.
 	\begin{lemma}\label{squaretest}
 		Let $\po\equiv\pt \equiv 1 \pmod 4$ be two prime numbers.   
 		\begin{enumerate}[$1)$]
 			\item If $N(\varepsilon_{2\po\pt})=-1$, then $\sqrt{\varepsilon_{2}\varepsilon_{\po}\varepsilon_{\pt}\varepsilon_{2\po\pt}}\in \KK.$
 			\item If $N(\varepsilon_{\po\pt})=-1$, then $\sqrt{\varepsilon_{\po}\varepsilon_{\pt}\varepsilon_{\po\pt}}\in \KK$.
 		\end{enumerate}
 	\end{lemma}
 	\begin{proof}
 		\begin{enumerate}[$1)$]
 			\item	Since $ \po \equiv\pt \equiv 1 \pmod 4$, this implies that  there exist  $\pi_1,\pi_2,\pi_3,\pi_4\in \ZZ[i]$ such that $\po=\pi_1\cdot\pi_2$ and $\pt=\pi_3\cdot\pi_4$ with $\overline{\pi_1}=\pi_2$ and $\overline{\pi_3}=\pi_4$. 
 			Put  $\varepsilon_{\po}=x+y\sqrt{\po}\in \ZZ[\frac{1+\sqrt{\po}}{2}]$ and $\varepsilon_{\pt}=a+b\sqrt{\pt}\in \ZZ[\frac{1+\sqrt{\pt}}{2}]$,
 			where $x$, $y$, $a$ and $b$ are   integers or semi-integers. We distinguish two cases:\\
 			$\star$ Assume that  $x$, $y$, $a$ and $b$ are are integers. As $N(\varepsilon_{\po})=-1$, we obtain the following systems:
 			$$(1):\ \left\{ \begin{array}{ll}
 				x\pm i=y_1^2\\
 				x\mp i=\pi_1\pi_2 y_2^2,
 			\end{array}\right. \quad
 			\text{or}\quad
 			(2):\ \left\{ \begin{array}{ll}
 				x\pm i=\pi_1y_1^2\\
 				x\mp i=\pi_2 y_2^2,
 			\end{array}\right.\quad
 			\text{or}\quad
 			(3):\ \left\{ \begin{array}{ll}
 				x\pm i=i\pi_1y_1^2\\
 				x\mp i=-i\pi_2 y_2^2,
 			\end{array}\right.\quad
 			$$	
 			where $y_1$ and $y_2$ are two elements of  $\ZZ[i]$ such that  $y=y_1 y_2$.
 			The first system  does not hold, in fact it implies that  $\sqrt{2\varepsilon_{\po}}\in \ZZ[\sqrt{\po}]$, which is not true.
 			Notice that the systems (2) and (3) give respectively the following equalities:
 			\begin{eqnarray}
 				\sqrt{2\varepsilon_{\po}}&=&\sqrt{\pi_1}y_1 + \sqrt{\pi_2}y_2, \text{ and }\label{eq1}\\
 				2\sqrt{\varepsilon_{\po}}&=&\sqrt{i\pi_1}y_1+\sqrt{-i\pi_2}y_2=(1+i)\sqrt{\pi_1}y_1 + (1-i)\sqrt{\pi_2}y_2.\nonumber
 			\end{eqnarray}
 			Notice furthermore that we have similar equalities for $\varepsilon_{\pt}$, i.e.
 			\begin{eqnarray}
 				\sqrt{2\varepsilon_{\pt}}&=&\sqrt{\pi_3}b_1 + \sqrt{\pi_4}b_2 \label{eq2}\text{ and }\\  
 				2\sqrt{\varepsilon_{\pt}}&=&(1+i)\sqrt{\pi_3}b_1 + (1-i)\sqrt{\pi_4}b_2. \label{eq23} 
 			\end{eqnarray} 
 			with $b_1$, $b_2\in \ZZ[i] $ such that $b=b_1b_2$.
 			Similarly, if $\alpha$ and $\beta $ are such that   $\varepsilon_{2\po\pt}=\alpha+\beta\sqrt{2\po\pt}$, then  according to 
 			\cite[Lemma 3.1]{AzZektaous}, $\sqrt{\varepsilon_{2\po\pt}}$ takes one of the following forms:
 			\begin{eqnarray}
 				\sqrt{\varepsilon_{2\po\pt}}&=&\frac{1}{2}(\beta_1(1+i)\sqrt{(1\pm i)\pi_1\pi_3}+\beta_2(1-i)\sqrt{(1\mp i)\pi_2\pi_4}),\label{eq3}\\
 				\sqrt{\varepsilon_{2\po\pt}}&=&\frac{1}{2}(\beta_1(1+i)\sqrt{(1\pm i)\pi_1\pi_4}+\beta_2(1-i)\sqrt{(1\mp i)\pi_2\pi_3}). \nonumber
 			\end{eqnarray}
 			Furthermore, we have $\sqrt{\varepsilon_2}=\frac{1}{2}(\sqrt{1+i}+\sqrt{1-i})$. We shall proceed case by case as follows:
 			\begin{enumerate}[$\bullet$]
 				\item If the equalities \eqref{eq1}, \eqref{eq2} and \eqref{eq3} hold, then we have 
 				{\footnotesize\begin{eqnarray*}
 						& & \sqrt{\varepsilon_2}\sqrt{2\varepsilon_{\po}} \sqrt{2\varepsilon_{\pt}}\sqrt{\varepsilon_{2\po\pt}} \\
 						&= &
 						\frac{1}{4}(\sqrt{\pi_1}y_1 + \sqrt{\pi_2}y_2)(\sqrt{\pi_3}b_1 + \sqrt{\pi_4}b_2) (\beta_1(1+i)\sqrt{(1+i)\pi_1\pi_3} + \beta_2(1-i)\sqrt{(1-i)\pi_2\pi_4})(\sqrt{1+i}+\sqrt{1-i})\\
 						&=& \frac{1}{4}( \sqrt{\pi_1 \pi_3} \, y_1 b_1 + \sqrt{\pi_1 \pi_4} \, y_1 b_2 + \sqrt{\pi_2 \pi_3} \, y_2 b_1 + \sqrt{\pi_2 \pi_4} \, y_2 b_2)( 
 						\beta_1(1+i)^2\sqrt{\pi_1\pi_3} \\
 						&&+ \beta_1(1+i)\sqrt{2\pi_1\pi_3}  
 						+ \beta_2(1-i)\sqrt{2\pi_2\pi_4} \\ 
 						&&+ \beta_2(1-i)^2\sqrt{\pi_2\pi_4})\\
 						&=&
 						\frac{1}{4}( y_1 b_1 \beta_1 2i \pi_1 \pi_3 +  y_1 b_1  \beta_1 (1+i) \pi_1 \pi_3 \sqrt{2} +  y_1 b_1 \beta_2 (1-i) \sqrt{2\po\pt}  -2i y_1 b_1  \beta_2  \sqrt{\po\pt}  \\
 						& &+   2i y_1 b_2 \beta_1  \pi_1 \sqrt{\pt} +  y_1 b_2  \beta_1 (1+i)\pi_1 \sqrt{2\pt} +  y_1 b_2  \beta_2 (1-i)\pi_4 \sqrt{2\po} -2i y_1 b_2  \beta_2 \pi_4\sqrt{\po} \\
 						&&+  2i y_2 b_1  \beta_1  \pi_3\sqrt{\po } + y_2 b_1  \beta_1 (1+i) \pi_3\sqrt{2\po} +  y_2 b_1  \beta_2 (1-i) \pi_2\sqrt{2\pt}   -2iy_2 b_1  \beta_2  \pi_2\sqrt{ \pt}  \\
 						& &+  2i y_2 b_2  \beta_1  \sqrt{\po\pt} +  y_2 b_2  \beta_1 (1+i) \sqrt{2\po\pt} +  y_2 b_2  \beta_2 (1-i) \pi_2 \pi_4\sqrt{2} -2i\pi_2 \pi_4 y_2 b_2 \beta_2 )  \\
 						&=&  A+B\sqrt{2}+C\sqrt{\po}+D\sqrt{\pt}+E\sqrt{2\po}+F\sqrt{2\pt}+G\sqrt{\po\pt}+H\sqrt{2\po\pt}\in \KK,
 				\end{eqnarray*}}
 			since the
 			expression above is the sum of pairwise conjugated complex numbers.
 				
 				\item If the equalities \eqref{eq1}, \eqref{eq23} and \eqref{eq3} hold, then we have 
 				$\sqrt{\varepsilon_{2}}\sqrt{2\varepsilon_{\po}}2\sqrt{\varepsilon_{\pt}}\sqrt{\varepsilon_{2\po\pt}}=\frac{1}{4}(\sqrt{1+i}+\sqrt{1-i})(\sqrt{\pi_1}y_1 + \sqrt{\pi_2}y_2)((1+i)\sqrt{\pi_3}b_1 + (1-i)\sqrt{\pi_4}b_2)((\beta_1(1+i)\sqrt{(1+ i)\pi_1\pi_3}+\beta_2(1-i)\sqrt{(1- i)\pi_2\pi_4}))\in \KK$ as above.
 				\item The verification of the other cases is analogous.
 			\end{enumerate}
 			\noindent	$\star$ When $x$, $y$, $a$ and $b$ are are semi-integers such that $x=\tilde{x}/2$, $y=\tilde{y}/2$, $a=\tilde{a}/2$ and $b=\tilde{b}/2$ for some integers $\tilde{x}$, $\tilde{y}$, $\tilde{a}$ and $\tilde{b}$, then     the conditions on the norm imply that  $(\tilde{x}+ 2i)(\tilde{x}-2i)=\po \tilde{y}^2$ and $(\tilde{a}+ 2i)(\tilde{a}-2i)=\pt \tilde{b}^2$,
 			which gives the same systems as above and so the same results. Thus, we have the first item. 
 			\item 	Assume that $N(\varepsilon_{\po\pt})=-1$. Let   $\varepsilon_{\po\pt}=c+d\sqrt{\po\pt}$, where $c$ and $d$ are integers or semi-integers. According to \cite[Remark 2.3]{AzZektaous}, we have:
 			$$
 			\left \{
 			\begin{array}{ccc}
 				\sqrt{\varepsilon_{\po\pt}} & = & d_1\sqrt{\pi_1\pi_3}+d_2\sqrt{\pi_2\pi_4}\; or \\
 				\sqrt{\varepsilon_{\po\pt}} & = & d_1\sqrt{\pi_1\pi_4}+d_2\sqrt{\pi_2\pi_3}\; or \\
 				\sqrt{2\varepsilon_{\po\pt}} & = & d_1\sqrt{\pi_1\pi_3}+d_2\sqrt{\pi_2\pi_4}\; or \\
 				\sqrt{2\varepsilon_{\po\pt}} & = & d_1\sqrt{\pi_1\pi_4}+d_2\sqrt{\pi_2\pi_3}\; 
 			\end{array}
 			\right.
 			$$
 			where $d_1$ and $d_2$ are two elements of $\ZZ[i]$ such that $d=d_1d_2$.	Proceeding similarly as in the proof of the first item we prove that $\sqrt{\varepsilon_{\po}\varepsilon_{\pt}\varepsilon_{\po\pt}}\in \KK$.

 		\end{enumerate}	
 	\end{proof}

 	\begin{lemma}\label{norms}
 		Let $p\equiv p' \equiv 1 \pmod 4$ be two prime numbers.  
 		\begin{enumerate}[$1)$]
 			\item $N(\varepsilon_{p})=-1$.
 			\item If $p\equiv 5 \pmod 8$, then  $N(\varepsilon_{2p})=-1$  $($cf.   \cite[Proposition 19.9 ,  p. 147]{connor88}$)$.
 			\item If $\genfrac(){}{0}{p}{p'} =-1$, then $N(\varepsilon_{p p'})=-1$ $($cf.   \cite[Proposition 19.9 ,  p. 147]{connor88}$)$.
 			\item If two of the set $\left\{\genfrac(){}{0}{p}{p'},\genfrac(){}{0}{2}{p}, \genfrac(){}{0}{2}{p'}\right\}$ are equal to -1 , then $N(\varepsilon_{2p p'})=-1$  $($cf.   \cite[Corollary $3.6$,  p. 98]{AziTaouss}$)$.
 			\item If $p\equiv 1\pmod 8$ and  $  \genfrac(){}{0}{2}{p}_4\not=  \genfrac(){}{0}{p}{2}_4$, then $N(\varepsilon_{2p})=1$ $($cf.   \cite[Theorem 2 (2)(a)]{kuvcera1995parity}$)$. 
 			\item If  $p\equiv 1\pmod 8$ and  $  \genfrac(){}{0}{2}{p}_4 =  \genfrac(){}{0}{p}{2}_4=-1$, then $N(\varepsilon_{2p})=-1$ $($cf.   \cite[Theorem 2 (2)(b)]{kuvcera1995parity}$)$. 
 		\end{enumerate}
 	\end{lemma}
 	

 	\section{\bf The case: $\po \equiv \pt\equiv 5 \pmod8$ }$\,$ 		\label{casestable}	
 	
 	Let us start by proving some useful facts.
 	\begin{lemma}\label{equalies}
 		Let  $\po $ and $ \pt$ be two primes such that $\po \equiv \pt\equiv 5 \pmod8$    and $\genfrac(){}{0}{\po}{\pt} =-1$. Then
 		$h_2(k_3)=h_2(k_5)=h_2(k_6)$ and $h_2(\KK)=\frac 12h_2(k_3)=q(k_3)$.
 	\end{lemma}	
 	\begin{proof}Let $i\in \{3,5,6\}$. Put $F=\QQ(\sqrt{2})$, $\QQ(\sqrt{p_2})$ or $\QQ(\sqrt{p_1})$ according to whether $i=3$, $5$ or $6$. 
 		Notice that the class number of $ F$ is odd (cf. Lemma \ref{class numbers of quadratic field} (1)), then by Lemma \ref{ambiguous class number formula}, we have  $r_2(\cl(k_i))=	t_{k_i/F}-1-e_{k_i/F}$. From our conditions, we get 	$t_{k_i/F}=2$. Therefore, 	$r_2(\cl(k_i))=	1-e_{k_i/F}$. Noting that $\KK$ is an unramified quadratic extension of $k_i$, then the class number of $k_i$ is even. Thus necessarily $e_{k_i/F}=0$ and so the $\cl_2(k_i)$ is cyclic (cf.  Lemmas \ref{wada's f.} and \ref{class numbers of quadratic field} ).
 	\end{proof}	
 	\begin{remark}\label{lemma2}
 		According to the above proof, if $\po $ and $ \pt$ be two primes such that $\po \equiv \pt\equiv 5 \pmod8$    and $\genfrac(){}{0}{\po}{\pt} = 1$, then 
 		$h_2(\KK)=\frac 12h_2(k_3)$.
 	\end{remark}

 	\begin{corollary}\label{sqrtcnd}
 		Keep the same hypothesis as in the above lemma. Then $h_2(k_i)=2q(k_i)$ and the following assertions are equivalent:
 		\begin{enumerate}[$1)$]
 			\item $ \sqrt{ \varepsilon_{2}\varepsilon_{\po\pt}\varepsilon_{2\po\pt}}\in k_3$, 
 			
 			\item  $ \sqrt{\varepsilon_{\pt}\varepsilon_{2\po}\varepsilon_{2\po\pt}}\in k_5$,
 			
 			\item  $\sqrt{  \varepsilon_{\po}\varepsilon_{2\pt}\varepsilon_{2\po\pt}}\in k_6$,
 			\item $\sqrt{\varepsilon_{2 \po}    \varepsilon_{2 \pt}    \varepsilon_{  \po\pt}}\in k_7$.
 		\end{enumerate}
 	\end{corollary}			
 	\begin{proof} Let $i\in \{3,5,6\}$.   Using Lemma  \ref{wada's f.} and 
 		Lemma \ref{class numbers of quadratic field}, one can easily check that	$h_2(k_i)=2q(k_i)$. As by the previous lemma we have $h_2(k_3)=h_2(k_5)=h_2(k_6)$, then 
 		$q(k_3)=2$ if and only if $q(k_5)=2$, if and only if $q(k_6)=2$. So   by Lemmas \ref{lmunit} and \ref{biquadunutk6}, the first three items are equivalent. Let us now check that $1)$ and $4)$ are equivalent. By Lemma  \ref{wada's f.} and 
 		Lemma \ref{class numbers of quadratic field}, we deduce that $h_2(k_7)=2q(k_7)$. As $\KK/k_7$ is unramified, then by class field theory, $q(k_7)=1$ if and only if $h_2(\KK)=1$, which is equivalent to $q(k_3)=1$ (cf. Lemma \ref{equalies}). So the result.
 	\end{proof}

 	\begin{remark}\label{forq4} As $N(\varepsilon_{\po})=N(\varepsilon_{\pt})=N(\varepsilon_{\po\pt})=-1$, then according to \cite{Ku-43} a F.S.U of $k_4=\QQ(\sqrt{\po},\sqrt{\pt})$ is one of the following $\{\varepsilon_{\po},\varepsilon_{\pt},\varepsilon_{\po\pt} \}$ or $\{\varepsilon_{\po},\varepsilon_{\pt},\sqrt{\varepsilon_{\po}\varepsilon_{\pt}\varepsilon_{\po\pt}}\}$. 
 		Since the 2-class number of $\QQ(\sqrt{\po\pt})$ equal to 2 and $k_4$ is an unaramified quadratic extension of $\QQ(\sqrt{\po\pt})$, we have $h(k_4)$ is odd.\\
 		On the other hand by Lemma   \ref{class numbers of quadratic field} we have :
 		\begin{eqnarray}
 			1=h_2(k_4)=\frac{1}{4}q(k_4)h_2(\po)h_2(\pt)h_2(\po\pt),
 		\end{eqnarray}
 		so $q(k_4)=2$.  It follows that $\{\varepsilon_{\po},\varepsilon_{\pt},\sqrt{\varepsilon_{\po}\varepsilon_{\pt}\varepsilon_{\po\pt}}\}$ is a F.S.U of $k_4$.
 	\end{remark}

 	\begin{theorem}\label{MT1A} Let $\po \equiv \pt\equiv 5 \pmod8$ be two prime numbers and   $\KK=\QQ(\sqrt 2, \sqrt{\po}, \sqrt{\pt} )$. We have:
 		\begin{enumerate}[\rm $1)$]
 			\item  If $N(\varepsilon_{\po\pt})=1$,  then
 			\begin{enumerate}[\rm $\bullet$]
 				
 				\item The unit group of $\KK$ is :
 				$$E_{\KK}=\langle -1, \varepsilon_2,\varepsilon_{\po}, \varepsilon_{\pt}, \sqrt{\varepsilon_2\varepsilon_{\po}\varepsilon_{2\po}},\sqrt{\varepsilon_2\varepsilon_{\pt}\varepsilon_{2\pt}}, \sqrt{\varepsilon_{\po\pt}},\sqrt{\varepsilon_2\varepsilon_{\po}\varepsilon_{\pt}\varepsilon_{2\po\pt}} \rangle.$$
 				\item The $2$-class number of  $\KK$ is $\frac{1}{2}h_2(\po\pt)$
 			\end{enumerate}
 			
 			\item If   $N(\varepsilon_{\po\pt})=-1$ and $\varepsilon_{2}\varepsilon_{\po\pt}\varepsilon_{2\po\pt}$ is \textbf{not} a square in $k_3$,  then
 			
 			\begin{enumerate}[\rm $\bullet$]
 				\item  The unit group of $\KK$ is :
 				$$E_{\KK}=\langle -1,  \varepsilon_{2}, \varepsilon_{\po},   \varepsilon_{\pt}, \sqrt{\varepsilon_{2}\varepsilon_{\po}\varepsilon_{2\po}}, 
 				\sqrt{\varepsilon_{2}\varepsilon_{\pt} \varepsilon_{2\pt}}, \sqrt{\varepsilon_{2}\varepsilon_{\po\pt}\varepsilon_{2\po\pt}},\sqrt{\varepsilon_{\po}\varepsilon_{\pt}\varepsilon_{\po\pt}}  \rangle.$$
 				
 				\item  The $2$-class number of  $\KK$ is   $\frac{1}{2}h_2(\po\pt)$.
 			\end{enumerate}

 			\item If  $N(\varepsilon_{\po\pt})=-1$ and $\varepsilon_{2}\varepsilon_{\po\pt}\varepsilon_{2\po\pt}$ is a square in $k_3$,  then
 			\begin{enumerate}[\rm $\bullet$]
 				
 				\item The unit group of $\KK$ is :
 				\begin{eqnarray*}
 					E_{\KK}=\langle -1,  \varepsilon_{2}, && \varepsilon_{\po},\varepsilon_{\pt}, \sqrt{\varepsilon_{2}\varepsilon_{\po}\varepsilon_{2\po}},\sqrt{\varepsilon_{2}\varepsilon_{\pt}\varepsilon_{2\pt}}, \sqrt{\varepsilon_{\po}\varepsilon_{\pt}\varepsilon_{\po\pt}},\\
 					&& \sqrt{\varepsilon_{2}^a\varepsilon_{\po}^b \varepsilon_{\pt}^c \varepsilon_{\po\pt}^d \sqrt{\varepsilon_{2}\varepsilon_{\po}\varepsilon_{2\po}}\sqrt{\varepsilon_{2}\varepsilon_{\pt}\varepsilon_{2\pt}}
 						\sqrt{\varepsilon_{2}\varepsilon_{\po\pt}\varepsilon_{2\po\pt} }}   \rangle,
 				\end{eqnarray*}
 				for $a$,$b$,$c$ and $d$ satisfying the equalities (8)-(14) shown in the proof.
 				\item The $2$-class class number of  $\KK$ is    $h_2(p_1p_2)$.
 			\end{enumerate}
 		\end{enumerate}
 	\end{theorem}
 	\begin{proof}
 		We shall use    the preliminaries and the method presented in  Section \ref{sec2prep}. Therefore, we need the unit groups of the three intermediate fields $k_i=\mathbb{Q}(\sqrt{2},\sqrt{p_i})$,  for $i\in \{1,2\}$ and $k_3=\mathbb{Q}(\sqrt{2},\sqrt{\po\pt})$ (cf. Figure \ref{fig:I}).
 		Notice that under our assumptions, we have $N(\varepsilon_{2\po})=N(\varepsilon_{\pt})=N(\varepsilon_{\po})=N(\varepsilon_{2\pt})=N(\varepsilon_{2\po\pt}) =-1$ (cf. Lemma \ref{norms}). Thus  by Lemma \ref{fork1k2}, for $i\in\{1,2\}$,  
 		a F.S.U of $k_i=\mathbb{Q}(\sqrt{2},\sqrt{p_i})$,  for $i\in \{1,2\}$, is given by $\{\varepsilon_{2}, \varepsilon_{p_i},	\sqrt{\varepsilon_{2}\varepsilon_{p_i}\varepsilon_{2p_i}}\}$. Furthermore, by Lemma 
 		\ref{lmunit} (1)-(2),  a F.S.U of $k_3=\mathbb{Q}(\sqrt{2},\sqrt{\po\pt})$ is : 
 		\begin{enumerate}[$\star$]
 			\item   $\{\varepsilon_{2}, \varepsilon_{\po\pt},	\sqrt{\varepsilon_{2}\varepsilon_{\po\pt}\varepsilon_{2\po\pt}}\}$,
 			if $N(\varepsilon_{ \po\pt}) =-1$ and $\varepsilon_{2}\varepsilon_{\po\pt}\varepsilon_{2\po\pt}$ is a square in $k_3$, 
 			\item   $\{\varepsilon_{2}, \varepsilon_{\po\pt},\varepsilon_{2\po\pt}\}$ else.
 		\end{enumerate}
 		Let us start by proving the first item. 
 		\begin{enumerate}[$1)$]
 			\item Assume that $N(\varepsilon_{\po\pt})=1$.	We have: 	$$E_{k_1}E_{k_2}E_{k_3}=\langle-1,  \varepsilon_{2}, \varepsilon_{\po},\varepsilon_{\pt}, \varepsilon_{\po\pt} , \varepsilon_{2\po\pt}, \sqrt{\varepsilon_{2}\varepsilon_{\po}\varepsilon_{2\po}},\sqrt{\varepsilon_{2}\varepsilon_{\pt}\varepsilon_{2\pt}}\rangle.$$	
 			Let us find the elements $\chi$   of $\KK$ which are the  square root of an element of $E_{k_1}E_{k_2}E_{k_3}$. Therefore, we can assume that
 			$$\chi^2=\varepsilon_{2}^a\varepsilon_{\po}^b \varepsilon_{\pt}^c\varepsilon_{\po\pt}^d\varepsilon_{2\po\pt}^e\sqrt{\varepsilon_{2}\varepsilon_{\po}\varepsilon_{2\po}}^f\sqrt{\varepsilon_{2}\varepsilon_{\pt}\varepsilon_{2\pt}}^g
 			,$$
 			where $a, b, c, d, e, f$ and $g$ are in $\{0, 1\}$. Notice that now,  we transformed the problem of the computation of the unit group of $\KK$ into the problem of 
 			whether the above family of equations admits solutions  in $\KK$ or not.
 			We shall use the norm maps to eliminate  the equations  which do not occur.\\  
 			\noindent\ding{229} By   applying the norm $N_{\KK/k_1}=1+\tau_3$, we get:
 			\begin{eqnarray*}
 				N_{\KK/k_1}(\chi^2)&=& \varepsilon_{2}^{2a} \cdot\varepsilon_{\po}^{2b}\cdot (-1)^c\cdot (1)^d\cdot (-1)^e\cdot (\varepsilon_{2}\varepsilon_{\po} \varepsilon_{2\po} )^f   (-1)^{a_1 g }\cdot\varepsilon_{2}^g \\
 				&=& \varepsilon_{2}^{2a} \varepsilon_{\po}^{2b}(\varepsilon_{2}\varepsilon_{\po} \varepsilon_{2\po} )^f (-1)^{c+e+b_1 g}\varepsilon_{2}^g,
 			\end{eqnarray*}
 			with $b_1 \in \{0,1\}$. We conclude that 
 			$c+e+b_1 g\equiv 0 \pmod 2$ and $g=0$. Thus, $c=e$ and 
 			$$\chi^2=\varepsilon_{2}^a\varepsilon_{\po}^b \varepsilon_{\pt}^c\varepsilon_{\po\pt}^d\varepsilon_{2\po\pt}^c\sqrt{\varepsilon_{2}\varepsilon_{\po}\varepsilon_{2\po}}^f 
 			.$$
 			\noindent\ding{229} By   applying the norm $N_{\KK/k_2}=1+\tau_2$, we get :
 			\begin{eqnarray*}
 				N_{\KK/k_1}(\chi^2)&=& \varepsilon_{2}^{2a}\cdot (-1)^b \cdot\varepsilon_{\pt}^{2c} \cdot(1)^d\cdot (-1)^c\cdot    (-1)^{c_1 f }\cdot\varepsilon_{2}^f \\
 				&=& \varepsilon_{2}^{2a} \varepsilon_{\pt}^{2c}(-1)^{b+c+c_1 f}\varepsilon_{2}^f,
 			\end{eqnarray*}
 			with $c_1 \in \{0,1\}$. Thus, 
 			$b+c+c_1 f\equiv 0 \pmod 2$ and $f=0$. Therefore $b=c$ and 
 			$$ \chi^2=\varepsilon_{2}^{a}\varepsilon_{\po}^{b}\varepsilon_{\pt}^{b}\varepsilon_{\po\pt}^{d}\varepsilon_{2\po\pt}^{b}     .$$
 			
 			\noindent\ding{229} Now let us   apply the norm $N_{\KK/k_5}=1+\tau_1\tau_2$. We have:
 			\begin{eqnarray*}
 				N_{\KK/k_5}(\chi^2)&=& (-1)^{a} \cdot(-1)^b\cdot\varepsilon_{\pt}^{2b}\cdot  (1)^d \cdot\varepsilon_{2\po\pt}^{2b}.\\
 				&=& \varepsilon_{\pt}^{2b}\varepsilon_{2\po\pt}^{2b}(-1)^{a+b}
 			\end{eqnarray*}
 			Thus $a=b$. Hence
 			$$ \chi^2=(\varepsilon_{2}\varepsilon_{\po}\varepsilon_{\pt}\varepsilon_{2\po\pt})^{a} \varepsilon_{\po\pt}^{d}.    $$
 			Notice that according to the proof of \cite[Proposition 2.5]{AzZektaous}, $\sqrt{\varepsilon_{\po\pt}}\in \KK$. So we can assume that
 			\begin{eqnarray*}
 				\chi^2=(\varepsilon_{2}\varepsilon_{\po}\varepsilon_{\pt}\varepsilon_{2\po\pt})^{a}  . 
 			\end{eqnarray*}  
 			On the other hand, by Lemma \ref{squaretest}, we have  
 			$\sqrt{\varepsilon_{2}\varepsilon_{\po}\varepsilon_{\pt}\varepsilon_{2\po\pt}}\in \KK$. Thus we have the first item.
 			
 			\item Assume that  $N(\varepsilon_{  \po\pt})=-1$ and  $\varepsilon_{2}\varepsilon_{\po\pt}\varepsilon_{2\po\pt}$ is not a square in $k_3$. So we have:
 			$$E_{k_1}E_{k_2}E_{k_3}=\langle-1,  \varepsilon_{2}, \varepsilon_{\po},\varepsilon_{\pt}, \varepsilon_{\po\pt} ,\varepsilon_{2\po\pt},  \sqrt{\varepsilon_{2}\varepsilon_{\po}\varepsilon_{2\po}},\sqrt{\varepsilon_{2}\varepsilon_{\pt}\varepsilon_{2\pt}} \rangle.$$
 			
 			Following the    Wada's method, let us look for  the elements $\chi$  of $\KK$  (if any) which are the square root of an element  of  $E_{k_1}E_{k_2}E_{k_3}$. We  can assume that :
 			\begin{eqnarray}\label{chi1}
 				\chi^2=\varepsilon_{2}^a\varepsilon_{\po}^b \varepsilon_{\pt}^c\varepsilon_{\po\pt}^d \varepsilon_{2\po\pt}^e\sqrt{\varepsilon_{2}\varepsilon_{\po}\varepsilon_{2\po}}^f\sqrt{\varepsilon_{2}\varepsilon_{\pt}\varepsilon_{2\pt}}^g,	 
 			\end{eqnarray} 
 			where $a, b, c, d, e, f$ and $g$ are in $\{0, 1\}$. \\
 			\noindent\ding{229} Let us start by applying the norm $N_{\KK/k_1}=1+\tau_3$, we get:
 			\begin{eqnarray*}
 				N_{\KK/k_1}(\chi^2)&=& \varepsilon_{2}^{2a}\cdot \varepsilon_{\po}^{2b}\cdot (-1)^c\cdot (-1)^d\cdot (-1)^e\cdot (\varepsilon_{2}\varepsilon_{\po} \varepsilon_{2\po} )^f\cdot   (-1)^{\alpha g }\cdot\varepsilon_{2}^g, \\
 				&=& \varepsilon_{2}^{2a} \varepsilon_{\po}^{2b}(\varepsilon_{2}\varepsilon_{\po} \varepsilon_{2\po} )^f(-1)^{c+d+e+\alpha g} \varepsilon_{2}^g,
 			\end{eqnarray*}
 			with $\alpha \in \{0,1\}$, and then we conclude that $c+d+e+\alpha g \equiv 0  \pmod 2 $, and $g=0$ because  $\varepsilon_{2}\varepsilon_{\po} \varepsilon_{2\po}$ is a square in $
 			k_1$  and    $\varepsilon_{2}$ is not. Thus, $c+d+e \equiv 0 \pmod 2$ and
 			$$\chi^2=\varepsilon_{2}^a\varepsilon_{\po}^b \varepsilon_{\pt}^c\varepsilon_{\po\pt}^d \varepsilon_{2\po\pt}^e\sqrt{\varepsilon_{2}\varepsilon_{\po}\varepsilon_{2\po}}^f 
 			.$$
 			
 			\noindent\ding{229} Now we apply the norm $N_{\KK/k_6}=1+\tau_1 \tau_3$, we get:
 			\begin{eqnarray*}
 				N_{\KK/k_6}(\chi^2)&=& (-1)^{a}\cdot\varepsilon_{\po}^{2b}\cdot (-1)^{c}\cdot(-1)^{d}\cdot\varepsilon_{2\po\pt}^{2e}\cdot (-1)^{\beta f}\cdot \varepsilon_{\po}^f,\\
 				&=&\varepsilon_{\po}^{2b}  \varepsilon_{2\po\pt}^{2e} (-1)^{a+c+d+\beta f} \varepsilon_{\po}^f,
 			\end{eqnarray*}
 			for some $\beta\in\{0,1\}$. Therefore,  $a+c+d+\beta f \equiv 0 \pmod 2$  and $f= 0$, in fact $\varepsilon_{\po}$ is not a square in $k_6$. Thus, 
 			$$\chi^2=\varepsilon_{2}^a\varepsilon_{\po}^b \varepsilon_{\pt}^c\varepsilon_{\po\pt}^d \varepsilon_{2\po\pt}^e,$$
 			with  	$a+c+d \equiv 0 \pmod 2.$ \\
 			\noindent\ding{229} By applying  the norm   $N_{\KK/k_3}=1+\tau_2\tau_3$, we get:
 			\begin{eqnarray*}
 				N_{\KK/k_3}(\chi^2)&=&\varepsilon_{2}^{2a}\cdot \varepsilon_{\po\pt}^{2d}\cdot \varepsilon_{2\po\pt}^{2e}\cdot(-1)^{b+c}.
 			\end{eqnarray*}
 			So $b=c$. Therefore, we have:
 			$$\chi^2=\varepsilon_{2}^a\varepsilon_{\po}^b \varepsilon_{\pt}^b\varepsilon_{\po\pt}^d \varepsilon_{2\po\pt}^e.$$
 			\noindent\ding{229} Let us apply the norm $N_{\KK/k_4}=1+\tau_1$. We have:
 			\begin{eqnarray*}
 				N_{\KK/k_4}(\chi^2)&=& \varepsilon_{\po}^{2b}\cdot \varepsilon_{\pt}^{2b}\cdot\varepsilon_{2\po\pt}^{2d}\cdot(-1)^{a+e},
 			\end{eqnarray*}
 			which gives $a=e$. Thus,  
 			$$\chi^2=\varepsilon_{2}^a\varepsilon_{\po}^b \varepsilon_{\pt}^b\varepsilon_{\po\pt}^d \varepsilon_{2\po\pt}^a
 			,$$
 			with  	$a+b+d \equiv 0 \pmod 2.$  So
 			if $a=0$ we get  $b=d$, and if  $a=1$, then  ($b=0$ and $d=1$) or ($b=1$ and $d=0$). Therefore, we eliminated all the equations \eqref{chi1} except the following:
 			$$\chi^2=\varepsilon_{\po} \varepsilon_{\pt}\varepsilon_{\po\pt} ,  \qquad   \chi^2=\varepsilon_{2} \varepsilon_{\po\pt} \varepsilon_{2\po\pt} , \qquad  \chi^2=\varepsilon_{2}\varepsilon_{\po} \varepsilon_{\pt} \varepsilon_{2\po\pt}.$$
 			By Remark \ref{remarkzekh} and Lemma \ref{squaretest},  we have 
 			$\sqrt{\varepsilon_{\po} \varepsilon_{\pt}\varepsilon_{\po\pt}}$,  $\sqrt{\varepsilon_{2} \varepsilon_{\po\pt} \varepsilon_{2\po\pt}}$, $\sqrt{\varepsilon_{2}\varepsilon_{\po} \varepsilon_{\pt} \varepsilon_{2\po\pt}} \in \KK$. Notice that $\sqrt{\varepsilon_{2}\varepsilon_{\po} \varepsilon_{\pt} \varepsilon_{2\po\pt}}=\varepsilon_{\po\pt}^{-1}\sqrt{\varepsilon_{\po} \varepsilon_{\pt}\varepsilon_{\po\pt}} \sqrt{\varepsilon_{2} \varepsilon_{\po\pt} \varepsilon_{2\po\pt}}$.
 			Hence,
 			$$E_{\KK}=\langle -1,  \varepsilon_{2}, \varepsilon_{\po},   \varepsilon_{\pt}, \sqrt{\varepsilon_{2}\varepsilon_{\po}\varepsilon_{2\po}}, 
 			\sqrt{\varepsilon_{2}\varepsilon_{\pt} \varepsilon_{2\pt}}, \sqrt{\varepsilon_{2}\varepsilon_{\po\pt}\varepsilon_{2\po\pt}},\sqrt{\varepsilon_{\po}\varepsilon_{\pt}\varepsilon_{\po\pt}}  \rangle.$$
 			And we still have $q(\KK)=2^4$.

 			\item Now assume that $N(\varepsilon_{  \po\pt})=-1$ and $\varepsilon_{2}\varepsilon_{\po\pt}\varepsilon_{2\po\pt}$ is a  square in $k_3$. We assume that $\left(\frac{p_1}{p_2}\right) =-1$.  So we have:
 			$$E_{k_1}E_{k_2}E_{k_3}=\langle-1,  \varepsilon_{2}, \varepsilon_{\po},\varepsilon_{\pt}, \varepsilon_{\po\pt} , \sqrt{\varepsilon_{2}\varepsilon_{\po}\varepsilon_{2\po}},\sqrt{\varepsilon_{2}\varepsilon_{\pt}\varepsilon_{2\pt}},\sqrt{\varepsilon_{2}\varepsilon_{\po\pt}\varepsilon_{2\po\pt} }\rangle.$$	
 			Proceeding as in the proof of the first item, we  consider
 			$\chi$   an element of $\KK$ of the form:
 			$$\chi^2=\varepsilon_{2}^a\varepsilon_{\po}^b \varepsilon_{\pt}^c\varepsilon_{\po\pt}^d\sqrt{\varepsilon_{2}\varepsilon_{\po}\varepsilon_{2\po}}^e\sqrt{\varepsilon_{2}\varepsilon_{\pt}\varepsilon_{2\pt}}^f
 			\sqrt{\varepsilon_{2}\varepsilon_{\po\pt}\varepsilon_{2\po\pt} }^g,$$
 			where $a, b, c, d, e, f$ and $g$ are in $\{0, 1\}$. 
 			Let us eliminate the  equations that do not occur.  
 			Let   $\sqrt{\varepsilon_{2}\varepsilon_{p_i}\varepsilon_{2p_i}}=\alpha_i+\beta_i \sqrt{2} +\gamma_i \sqrt{p_i} +\delta_i \sqrt{2p_i} $ (for $i=1$, $2$) and
 			$ \sqrt{\varepsilon_{2}\varepsilon_{\po\pt}\varepsilon_{2\po\pt} }=\alpha+\beta\sqrt{2} +\gamma\sqrt{\po\pt}+\delta\sqrt{2\po\pt}$, where 
 			$\alpha_i$, $\beta_i$, $\gamma_i$,  $\delta_i$, $\alpha$, $\beta$, $\gamma$ and   $\delta$ are rational numbers. Thus, we have the following table which collects some values of norm maps:
 			%
 			
 			\begin{center}
 				\begin{tabular}{|c|c|c|c|}
 					\hline
 					$\varepsilon$& $\sqrt{\varepsilon_{2}\varepsilon_{\po}\varepsilon_{2\po}}$&$\sqrt{\varepsilon_{2}\varepsilon_{\pt}\varepsilon_{2\pt}}$&$\sqrt{\varepsilon_{2}\varepsilon_{\po\pt}\varepsilon_{2\po\pt}}$ \\ 
 					\hline
 					$\varepsilon^{1+\tau_2}$ & $(-1)^{a_1}\varepsilon_{2}$ & $\varepsilon_{2}\varepsilon_{\pt}\varepsilon_{2\pt} $   &   $(-1)^{a_2}\varepsilon_{2}$   \\ 
 					\hline
 					$\varepsilon^{1+\tau_2\tau_3}$ & $(-1)^{a_1}\varepsilon_{2}$ & $(-1)^{b_2}\varepsilon_{2}$    &   $\varepsilon_{2}\varepsilon_{\po\pt}\varepsilon_{2\po\pt} $  \\ 
 					\hline
 					$\varepsilon^{1+\tau_1}$ & $(-1)^{c_1}\varepsilon_{\po}$ & $(-1)^{c_2}\varepsilon_{\pt}$    & $(-1)^{c_3}\varepsilon_{\po\pt}$  \\ 
 					\hline
 					$\varepsilon^{1+\tau_1\tau_3}$ & $(-1)^{c_1}\varepsilon_{\po}$ & $(-1)^{d_2}\varepsilon_{2\pt}$    & $(-1)^{d_3}\varepsilon_{2\po\pt}$  \\ 
 					\hline
 					$\varepsilon^{1+\tau_3}$ & $\varepsilon_{2}\varepsilon_{\po}\varepsilon_{2\po}$ & $(-1)^{b_2}\varepsilon_{2}$    & $(-1)^{a_2}\varepsilon_{2}$  \\ 
 					\hline
 					$\varepsilon^{1+\tau_1\tau_2}$ & $(-1)^{f_1}\varepsilon_{2\po}$ & $(-1)^{c_2}\varepsilon_{\pt}$    & $(-1)^{d_3}\varepsilon_{2\po\pt}$  \\ 
 					\hline
 					$\varepsilon^{1+\tau_1\tau_2\tau_3}$ & $(-1)^{f_1}\varepsilon_{2\po}$ & $(-1)^{d_2 }\varepsilon_{2\pt}$    & $(-1)^{c_3}\varepsilon_{\po\pt}$  \\ 
 					\hline
 				\end{tabular} 
 			\end{center}
 			where the letters  over $(-1)$ are elements of $\{0,1\}.$

 			\noindent\ding{229}   Let us start	by applying   the norm map $N_{\KK/k_2}=1+\tau_2$. We have
 			\begin{eqnarray*}
 				N_{\KK/k_2}(\chi^2)&=&
 				\varepsilon_{2}^{2a}(-1)^b \cdot \varepsilon_{\pt}^{2c}\cdot (-1)^d\cdot(-1)^{a_1 e} \cdot \varepsilon_{2}^{e} \cdot (\varepsilon_{2}\varepsilon_{\pt}\varepsilon_{2\pt})^{f}\cdot(-1)^{a_2 g}\cdot\varepsilon_{2}^{g}\\
 				&=& \varepsilon_{2}^{2a}\varepsilon_{\pt}^{2c}(\varepsilon_{2}\varepsilon_{\pt}\varepsilon_{2\pt})^f(-1)^{b+d+a_1 e+a_2 g} \varepsilon_{2}^{g+e}
 			\end{eqnarray*}
 			for some $a_1,a_2 \in \{0,1\}$.	Thus, $b+d+a_1 e +a_2g\equiv 0\pmod2$. As $\varepsilon_{2}$ is not a square in $k_2$, then $e=g$. Therefore, 
 			$$\chi^2=\varepsilon_{2}^a\varepsilon_{\po}^b \varepsilon_{\pt}^c\varepsilon_{\po\pt}^d\sqrt{\varepsilon_{2}\varepsilon_{\po}\varepsilon_{2\po}}^e\sqrt{\varepsilon_{2}\varepsilon_{\pt}\varepsilon_{2\pt}}^f
 			\sqrt{\varepsilon_{2}\varepsilon_{\po\pt}\varepsilon_{2\po\pt} }^e,$$
 			with $b+d+(a_1   +a_2)e\equiv 0\pmod2$.\\
 			\noindent\ding{229} By applying the norm $N_{\KK/k_3}=1+\tau_2\tau_3$, we get  
 			\begin{eqnarray*}
 				N_{\KK/k_3}(\chi^2)&=&\varepsilon_{2}^{2a}\cdot (-1)^b\cdot (-1)^c \cdot\varepsilon_{\po\pt}^{2d}\cdot (-1)^{b_1 e}\cdot\varepsilon_2^e \cdot(-1)^{b_2 f}\cdot\varepsilon_2^f \cdot(\varepsilon_{2}\varepsilon_{\po\pt}\varepsilon_{2\po\pt})^e \\
 				&=& 	\varepsilon_{2}^{2a}\varepsilon_{\po\pt}^{2d}  (\varepsilon_{2}\varepsilon_{\po\pt}\varepsilon_{2\po\pt})^e(-1)^{b+c+b_1 e+b_2 f} \varepsilon_{2}^{e+f}.
 			\end{eqnarray*}
 			with $b_1,b_2\in \{0,1\}$.
 			Thus $e=f$ and $b+c+b_1 e+b_2 f\equiv b+c+(b_1  +b_2 )e\equiv 0 \pmod 2$. Therefore, 
 			$$\chi^2=\varepsilon_{2}^a\varepsilon_{\po}^b \varepsilon_{\pt}^c\varepsilon_{\po\pt}^d\sqrt{\varepsilon_{2}\varepsilon_{\po}\varepsilon_{2\po}}^e\sqrt{\varepsilon_{2}\varepsilon_{\pt}\varepsilon_{2\pt}}^e
 			\sqrt{\varepsilon_{2}\varepsilon_{\po\pt}\varepsilon_{2\po\pt} }^e,$$ 
 			At this stage, one can deduce the result, but we continue applying norm maps to collect further information about the  exponents $a$, $b$, $c$ and $d$. \\
 			\noindent\ding{229} By applying the norm $N_{\KK/k_4}=1+\tau_1$, we get:
 			\begin{eqnarray*}
 				N_{\KK/k_4}(\chi^2)&=&(-1)^a\cdot  \varepsilon_{\po}^{2b}\cdot\varepsilon_{\pt}^{2c}\cdot\varepsilon_{\po\pt}^{2d}\cdot (-1)^{c_1 e}\cdot\varepsilon_{\po}^{e}\cdot(-1)^{c_2 e}\cdot\varepsilon_{\pt}^{e}\cdot(-1)^{c_3 e}\cdot\varepsilon_{\po\pt}^{e},\\
 				&=& \varepsilon_{\po}^{2b}\varepsilon_{\pt}^{2c}\varepsilon_{\po\pt}^{2d} (\varepsilon_{\po}\varepsilon_{\pt}\varepsilon_{\po\pt})^{e}(-1)^{a+(c_1+c_2+c_3)e},
 			\end{eqnarray*}
 			with $c_1,c_2$ and $c_3\in \{0,1\}$. Note that by Remark \ref{forq4}, $\varepsilon_{\po}\varepsilon_{\pt}\varepsilon_{\po\pt}$ is a square in $k_4$.  So the only information we get is that  $a+(c_1+c_2+c_3)e\equiv 0 \pmod 2$.\\ 
 			\noindent\ding{229} Now let us  apply the norm $N_{\KK/k_6}=1+\tau_1\tau_3$. We have
 			\begin{eqnarray*}
 				N_{\KK/k_6}(\chi^2)&=&(-1)^a \cdot \varepsilon_{\po}^{2b}\cdot(-1)^{c}\cdot(-1)^{d}\cdot (-1)^{d_1 e}\cdot\varepsilon_{\po}^{e}\cdot(-1)^{d_2 e}\cdot\varepsilon_{2\pt}^{e}\cdot(-1)^{d_3 e}\varepsilon_{2\po\pt}^{e},\\
 				&=& \varepsilon_{\po}^{2b} (\varepsilon_{\po}\varepsilon_{2\pt}\varepsilon_{2\po\pt})^{e}(-1)^{a+c+d+(d_1+d_2+d_3)e},
 			\end{eqnarray*}
 			with $d_1,d_2$ and $d_3\in \{0,1\}$. By Corollary \ref{sqrtcnd}, $\varepsilon_{\po}\varepsilon_{2\pt}\varepsilon_{2\po\pt}$ is a square in $k_6$. So the only information we deduce is $a+c+d+(d_1+d_2+d_3)e\equiv 0 \pmod 2$.\\ 
 			\noindent\ding{229} Now we apply the norm $N_{\KK/k_1}=1+\tau_3$. This gives:
 			\begin{eqnarray*}
 				N_{\KK/k_1}(\chi^2)&=&\varepsilon_2^{2a}\cdot  \varepsilon_{\po}^{2b} \cdot(\varepsilon_{2}\varepsilon_{\po}\varepsilon_{2\po})^{e} \cdot(-1)^{e_1 e}\cdot\varepsilon_{2}^e \cdot(-1)^{e_2 e}\cdot\varepsilon_{2}^e ,\\
 				&=&  \varepsilon_2^{2a+2e} \varepsilon_{\po}^{2b} (\varepsilon_{2}\varepsilon_{\po}\varepsilon_{2\po})^{e}(-1)^{c+d+(e_1+e_2)e}, 
 			\end{eqnarray*}
 			with $e_1,e_2 \in \{0,1\}$. As $\varepsilon_{2}\varepsilon_{\po}\varepsilon_{2\po}$ is a square in $k_1$, the only information we deduce is  that $c+d+(e_1+e_2)e\equiv 0 \pmod 2$.\\ 
 			\noindent\ding{229} Now we apply the norm $N_{\KK/k_5}=1+\tau_1\tau_2$. We have:
 			\begin{eqnarray*}
 				N_{\KK/k_5}(\chi^2)&=&(-1)^a\cdot(-1)^b \cdot\varepsilon_{\pt}^{2c}\cdot(-1)^d\cdot (-1)^{f_1 e} \cdot \varepsilon_{2\po}^e\cdot (-1)^{f_2 e}  \cdot\varepsilon_{\pt}^e \cdot(-1)^{f_3 e} \cdot \varepsilon_{2\po\pt}^e ,\\
 				&=& (\varepsilon_{\pt}\varepsilon_{2\po}\varepsilon_{2\po\pt})^e \varepsilon_{\pt}^{2c}(-1)^{c+b+d+(f_1+f_2+f_3)e},
 			\end{eqnarray*}
 			with $f_1,f_2$ and  $f_3 \in \{0,1\}$. As by Corollary \ref{sqrtcnd}, $\varepsilon_{\pt}\varepsilon_{2\po}\varepsilon_{2\po\pt}$ is a square in $k_5$,  then the only fact we can deduce   is  $c+b+d+(f_1+f_2+f_3)e\equiv 0 \pmod 2$.  \\
 			\noindent\ding{229} By applying the norm $N_{\KK/k_7}=1+\tau_1\tau_2\tau_3$, we get:
 			\begin{eqnarray*}
 				N_{\KK/k_7}(\chi^2)&=& (-1)^{a}\cdot (-1)^b\cdot (-1)^c\cdot   \varepsilon_{ \po\pt}^{2d}\cdot (-1)^{f_1e}  \varepsilon_{2 \po}^{e}\cdot (-1)^{d_2e}  \varepsilon_{2 \pt}^{e}\cdot (-1)^{c_3e}  \varepsilon_{  \po\pt}^{e} ,\\
 				&=& \varepsilon_{ \po\pt}^{2d} (\varepsilon_{2 \po}    \varepsilon_{2 \pt}    \varepsilon_{  \po\pt})^{e}(-1)^{a+b+c+f_1e+d_2e+c_3e}.
 			\end{eqnarray*}
 			So the only information we may deduce is that $	a+b+c+(f_1+d_2+c_3)e\equiv 0\pmod 2$.
 			Therefore, we eliminated all the equations except the following $$\chi^2=\varepsilon_{2}^a\varepsilon_{\po}^b \varepsilon_{\pt}^c\varepsilon_{\po\pt}^d\sqrt{\varepsilon_{2}\varepsilon_{\po}\varepsilon_{2\po}}^e\sqrt{\varepsilon_{2}\varepsilon_{\pt}\varepsilon_{2\pt}}^e
 			\sqrt{\varepsilon_{2}\varepsilon_{\po\pt}\varepsilon_{2\po\pt} }^e,$$
 			with    
 			\begin{eqnarray}
 				b+d+(a_1+a_2)e&\equiv& 0 \pmod 2,\label{eq7}\\
 				b+c+(a_1+b_2)e&=& 0 \pmod 2,\label{eq8}\\
 				a+(c_1+c_2+c_3)e&=& 0 \pmod 2,\label{eq9}\\
 				a+c+d+(c_1+d_2+d_3)e&=& 0 \pmod 2,\\
 				c+d+(b_2+a_2)e&=& 0 \pmod 2,\\
 				a+b+d+(f_1+c_2+d_3)e&=& 0 \pmod 2,\\
 				a+b+c+(f_1+d_2+c_3)e&=& 0 \pmod 2.\label{eq20}
 			\end{eqnarray}
 			\noindent\ding{51} If $e=0$, then $\eqref{eq9}$ implies $a=0$, and from $\eqref{eq7}$ and $\eqref{eq8}$ we get $b=c=d$, thus 
 			$$\chi^2=(\varepsilon_{\po} \varepsilon_{\pt} \varepsilon_{\po\pt})^b .$$
 			By Remark \ref{forq4}  $ \sqrt{\varepsilon_{\po} \varepsilon_{\pt} \varepsilon_{\po\pt}}\in \KK$.\\
 			\noindent\ding{51} Now if $e=1$, we have :
 			\begin{eqnarray}\label{eqq=5}
 				\chi^2=\varepsilon_{2}^a\varepsilon_{\po}^b \varepsilon_{\pt}^c \varepsilon_{\po\pt}^d (\sqrt{\varepsilon_{2}\varepsilon_{\po}\varepsilon_{2\po}}\sqrt{\varepsilon_{2}\varepsilon_{\pt}\varepsilon_{2\pt}}
 				\sqrt{\varepsilon_{2}\varepsilon_{\po\pt}\varepsilon_{2\po\pt} }).
 			\end{eqnarray}
 			Notice that if  $b=c=d=1$, then from the fact that $ \sqrt{\varepsilon_{\po} \varepsilon_{\pt} \varepsilon_{\po\pt}}\in \KK$, we can consider $\chi^2=\varepsilon_{2}^a  (\sqrt{\varepsilon_{2}\varepsilon_{\po}\varepsilon_{2\po}}\sqrt{\varepsilon_{2}\varepsilon_{\pt}\varepsilon_{2\pt}}
 			\sqrt{\varepsilon_{2}\varepsilon_{\po\pt}\varepsilon_{2\po\pt} })$.
 			On the other hand, by means of Lemmas \ref{class numbers of quadratic field} and \ref{wada's f.}, we get:
 			\begin{eqnarray*}
 				q(\KK)&=&\frac{2^{9}h_2(\KK) }{h_2(2) h_2(\po) h_2(\pt)h_2(2\po) h_2(2\pt)h_2(\po\pt)  h_2(2\po\pt)}\\ 
 				&=&\frac{2^{9}h_2(\KK) }{1\cdot 1 \cdot 1  \cdot 2 \cdot 2 \cdot 2 \cdot 4}=2^{4}h_2(\KK)=2^5.
 			\end{eqnarray*}
 			The last equality follows from Lemma \ref{equalies}.  So the equation \eqref{eqq=5} admits a solution in $\KK$, since otherwise we get $q(\KK)=2^4$. So we have  
 			$E_{\KK}=\langle -1,  \varepsilon_{2}, \varepsilon_{\po},\varepsilon_{\pt},\\ \sqrt{\varepsilon_{2}\varepsilon_{\po}\varepsilon_{2\po}},\sqrt{\varepsilon_{2}\varepsilon_{\pt}\varepsilon_{2\pt}}, \sqrt{\varepsilon_{\po}\varepsilon_{\pt}\varepsilon_{\po\pt}}, \sqrt{\varepsilon_{2}^a\varepsilon_{\po}^b \varepsilon_{\pt}^c \varepsilon_{\po\pt}^d (\sqrt{\varepsilon_{2}\varepsilon_{\po}\varepsilon_{2\po}}\sqrt{\varepsilon_{2}\varepsilon_{\pt}\varepsilon_{2\pt}}
 				\sqrt{\varepsilon_{2}\varepsilon_{\po\pt}\varepsilon_{2\po\pt} })}  \rangle,$
 			for some  $a$, $b$, $c$ and $d$ in $\{0,1\}$ with $(b,c,d)\not=(1,1,1)$
 		\end{enumerate} 
 		This completes the proof. 
 	\end{proof}
 	
 	\begin{remark} 
 		Notice that if $\left(\frac{p_1}{p_2}\right)=-1$, then the numbers $a$, $b$, $c$ and $d$ in the second item of the above theorem satisfy the equations \eqref{eq7}--\eqref{eq20}.
 	\end{remark}

 	\section{ \bf The case: $\po \equiv 1 \pmod8$, $  \pt\equiv 1 \pmod4$   and   $(n_3, n_4)\not=(-1,-1)$}$\,$ 		\label{section2}		
 	
 	In this section, we compute the unit group of $\KK=\QQ(\sqrt 2, \sqrt{p_1}, \sqrt{p_2} )$, where $\po \equiv 1 \pmod8$ and  $  \pt\equiv 1 \pmod4$  are two prime numbers such that $(n_3, n_4)\not=(-1,-1)$. Note that a possible interchanging of $p_1$ and $p_2$ is taken into account.
 	\begin{remark}\label{twosquare}
Keeping the previous notations, we have :
 		\begin{enumerate}[$\bullet$]
 			\item If  $N(\varepsilon_{\po\pt})=1$,  then $\varepsilon_{\po\pt}$ is a square in $\KK$.
 			
 			\item If $N(\varepsilon_{2p_j})=1$, then $\varepsilon_{2p_j}$ is not square in $k_j$ for $j\in \{5,6\}$.
 			
 			\item If $N(\varepsilon_{2\po\pt})=1$ and $x$ and $y$ are two integers such that $\varepsilon_{2\po\pt}=x+y\sqrt{2\po\pt}$, then $\sqrt{\varepsilon_{2\po\pt}}\in \KK$ and it takes one of the three following forms :
 			$$\frac{y_1}{2}\sqrt{2}+y_2\sqrt{\po\pt},  
 			\qquad y_1\sqrt{\pt}+\frac{y_2}{2}\sqrt{2\po} \;\;\textbf{or}\;\; y_1\sqrt{\po}+\frac{y_2}{2}\sqrt{2\pt},$$
 		according to whether $x\pm1$, $p_1(x\pm1)$ or $2p_1(x\pm1)$ is a square in $\NN$.	Where $y_1$ and $y_2$ are two integers such that $y=y_1 y_2$.

 		\end{enumerate}
 	\end{remark}
 	\begin{proof} Let us proof the first item  and  the  proof of the rest is analogous (or one can consult the proof of  \cite[Proposition 2.3]{AzZektaous} for the last item).
 			Assume that $N(\varepsilon_{\po\pt})=1$ and that there are integers  $a$ and $b$  such that $\varepsilon_{\po\pt}=a+b\sqrt{\po\pt}$. Thus $a^2-1=b^2\po\pt$. By   taking into account the unique factorization in $\ZZ$, we check that we have one of the following systems:
 			$$(1):\ \left\{ \begin{array}{ll}
 				a\pm 1=b_1^2\\
 				a\mp 1=p_1p_2 b_2^2,
 			\end{array}\right. \quad
 			\text{or}\quad
 			(2):\ \left\{ \begin{array}{ll}
 				a\pm 1=p_1b_1^2\\
 				a\mp 1=p_2 b_2^2,
 			\end{array}\right.\quad
 			\text{or}\quad
 			(3):\ \left\{ \begin{array}{ll}
 				a\pm 1=2p_1b_1^2\\
 				a\mp 1=2p_2 b_2^2,
 			\end{array}\right.\quad
 			$$	
 			where $b_1$ and $b_2$ are two integers such that  $b=b_1 b_2$ for   Systems $(1)$ and $(2)$, and $b=2b_1 b_2$ for    System $(3)$. From these systems we deduce that we have  :
 			 $$\sqrt{\varepsilon_{ \po\pt}}=\frac{b_1}{2}\sqrt{2}+\frac{b_2}{2}\sqrt{2\po\pt},  
 			\quad \frac{b_1}{2}\sqrt{2\po}+\frac{b_2}{2}\sqrt{2\pt} \;\;\;\textbf{or}\;\;\; b_1\sqrt{\po}+b_2\sqrt{\pt},$$ 
 			according to whether $a\pm1$, $p_1(a\pm1)$ or $2p_1(a\pm1)$ is a square in $\NN$.	
 			Thus,  $\sqrt{\varepsilon_{ \po\pt}}\in \KK$.\\
 			Now if  $a$ and $b$ are semi-integers, then  by considering   $a=\tilde{a}/2$ and $b=\tilde{b}/2$ where    $\tilde{a}$ and $\tilde{b}$ are integers, we proceed similarly to get the same result.  

 	\end{proof}

 	\begin{theorem}\label{MT3} Let $\po \equiv 1 \pmod8$, $  \pt\equiv 1 \pmod4$  be two prime numbers    such that   $(n_3, n_4)\not=(-1,-1)$. Let $x$ and  $y$ be two integers such that $\varepsilon_{2\po\pt}=x+y\sqrt{2\po\pt} $. Put $\KK=\QQ(\sqrt 2, \sqrt{\po}, \sqrt{\pt} )$.
 		\begin{enumerate}[\rm $1)$]
 			\item    Assume that  $(n_1,n_2,n_3, n_4)=(-1,-1-1,1)$. Then the unit group of $\KK$ is :
 			$$E_{\KK}=\langle -1,  \varepsilon_{2}, \varepsilon_{\po},   \varepsilon_{\pt},\sqrt{\varepsilon_{2\po\pt}}, \sqrt{\varepsilon_{2}\varepsilon_{\po}\varepsilon_{2\po}}, 
 			\sqrt{\varepsilon_{2}\varepsilon_{\pt} \varepsilon_{2\pt}}, \sqrt{\varepsilon_{\po}\varepsilon_{\pt}\varepsilon_{\po\pt}}  \rangle$$
 			
 			\item    Assume that $(n_1,n_2,n_3, n_4)=(-1, 1,-1,1)$. Then the unit group of $\KK$ is :
 			$$E_{\KK}=\langle-1,  \varepsilon_{2}, \varepsilon_{\po},\varepsilon_{\pt}, \sqrt{\varepsilon_{2\pt}}, \sqrt{\varepsilon_{2\po\pt}} ,  \sqrt{\varepsilon_{2}\varepsilon_{\po}\varepsilon_{2\po}}, \sqrt{\varepsilon_{\po}\varepsilon_{\pt}\varepsilon_{\po\pt}} \rangle.$$
 			
 			\item    Assume that $(n_1,n_2,n_3, n_4)=( 1, 1, -1, 1)$.    We have two cases: 
 			\begin{enumerate}[$a)$]
 				\item If $(x\pm1)$ is   {\bf not} a square in $\NN$,  then the unit group of $\KK$ is :
 				$$E_{\KK}=\langle-1,  \varepsilon_{2}, \varepsilon_{\po},\varepsilon_{\pt}, \sqrt{\varepsilon_{2\po}},\sqrt{\varepsilon_{2\pt}}, \sqrt{\varepsilon_{2\po\pt}},\sqrt{\varepsilon_{\po}\varepsilon_{\pt}\varepsilon_{\po\pt}} \rangle.$$
 				
 				\item Else,  then the unit group of $\KK$ is :   $$E_{\KK}=\langle -1, \varepsilon_{2}, \varepsilon_{\po},\varepsilon_{\pt}, \sqrt{\varepsilon_{2\po}}, \sqrt{\varepsilon_{2\pt}},    \sqrt{\varepsilon_{\po}\varepsilon_{\pt}\varepsilon_{\po\pt}},
 				\sqrt{\varepsilon_{2}^{\alpha a} \varepsilon_{\po}^{\alpha b}\varepsilon_{\pt}^{\alpha c}\varepsilon_{\po\pt}^{\alpha d} \sqrt{\varepsilon_{2\po}\varepsilon_{2\pt}\varepsilon_{2\po\pt}}^{1+\gamma}} \rangle.$$ 
 				for some  $\alpha $, $\gamma$, $a$, $b$, $c$ and $d$ in $\{0,1\}$ with $(b,c,d)\not=(1,1,1)$ and $\alpha \not= \gamma$.
 			\end{enumerate}
 			
 			\item    Assume that $(n_1,n_2,n_3, n_4)=(-1, 1, 1,-1)$. Then the unit group of $\KK$ is :
 			$$E_{\KK}=\langle-1,  \varepsilon_{2}, \varepsilon_{\po},\varepsilon_{\pt}, \sqrt{\varepsilon_{\po\pt}},\sqrt{\varepsilon_{2}\varepsilon_{\po}\varepsilon_{2\po}},\sqrt{\varepsilon_{2\pt}}, \sqrt{\varepsilon_{2}\varepsilon_{\po}\varepsilon_{\pt}\varepsilon_{2\po\pt}}    \rangle.$$
 			
 			\item     Assume that $(n_1,n_2,n_3, n_4)=(-1,-1, 1,-1)$. Then the unit group of $\KK$ is :
 			$$E_{\KK}=\langle-1,  \varepsilon_{2}, \varepsilon_{\po},\varepsilon_{\pt}, \sqrt{\varepsilon_{\po\pt}},\sqrt{\varepsilon_{2}\varepsilon_{\po}\varepsilon_{2\po}},\sqrt{\varepsilon_{2}\varepsilon_{\pt}\varepsilon_{2\pt}}, \sqrt{\varepsilon_{2}\varepsilon_{\po}\varepsilon_{\pt}\varepsilon_{2\po\pt}}    \rangle.$$
 			
 				\item    Assume that $(n_1,n_2,n_3, n_4)=( 1, 1, 1, -1)$. Then the unit group of $\KK$ is :
 			$$E_{\KK}=\langle-1,  \varepsilon_{2}, \varepsilon_{\po},\varepsilon_{\pt}, \sqrt{\varepsilon_{\po\pt}},\sqrt{\varepsilon_{2\po}}, \sqrt{\varepsilon_{2\pt}}, \sqrt{\varepsilon_{2}\varepsilon_{\po}\varepsilon_{\pt}\varepsilon_{2\po\pt}}   \rangle.$$
 			
 			\item    Assume that $(n_1,n_2,n_3, n_4)=(-1, 1, 1, 1)$. Then we have two cases :
 			\begin{enumerate}[$a)$]
 				\item If $(x\pm1)$ is a square in $\NN$ , then the units group of $\KK$ is 
 				$$E_{\KK}=\langle-1,  \varepsilon_{2}, \varepsilon_{\po},\varepsilon_{\pt},\sqrt{\varepsilon_{\po\pt}},  \sqrt{\varepsilon_{2\po\pt}} ,\sqrt{\varepsilon_{2\pt}} , \sqrt{\varepsilon_{2}\varepsilon_{\po}\varepsilon_{2\po}} \rangle.$$
 				\item Else, the units group of $\KK$ is 
 				$$E_{\KK}=\langle-1,  \varepsilon_{2}, \varepsilon_{\po},\varepsilon_{\pt}, \sqrt{\varepsilon_{\po\pt}},\sqrt{\varepsilon_{  2\po\pt}},\sqrt{\varepsilon_{2}\varepsilon_{\po}\varepsilon_{2\po}},
 				\sqrt{\varepsilon_{2}^{a\alpha}\varepsilon_{\pt}^{(1+a)\alpha}\sqrt{\varepsilon_{2\pt}}^{\gamma+1} }   \rangle.$$  
 				for    $\alpha$,  $\gamma$ and $a$  in $  \{0,1\}$ such that $\alpha\not=\gamma $ and $a \equiv u+1 \pmod 2$, where $u$ is the element of $\{0,1\}$ given by Lemma \ref{calcul} for $p=p_2$.  
 			\end{enumerate}

				\item    Assume that $(n_1,n_2,n_3, n_4)=(-1,-1,1,1)$. Then we have two cases :
				\begin{enumerate}[$a)$]
					\item If $(x\pm1)$ is a square in $\NN$, then the unit group of $\KK$ is 
						$$E_{\KK}=\langle-1,  \varepsilon_{2}, \varepsilon_{\po},\varepsilon_{\pt},\sqrt{\varepsilon_{\po\pt}},  \sqrt{\varepsilon_{2\po\pt}} ,  \sqrt{\varepsilon_{2}\varepsilon_{\po}\varepsilon_{2\po}}, \sqrt{\varepsilon_{\po}\varepsilon_{\pt}\varepsilon_{\po\pt}} \rangle.$$ 		
						
						\item Else, the unit group of $\KK$ is : 
						$$E_{\KK}=\langle-1,  \varepsilon_{2}, \varepsilon_{\po},\varepsilon_{\pt},\  \sqrt{\varepsilon_{\po\pt}},\sqrt{\varepsilon_{2}\varepsilon_{\po}\varepsilon_{2\po}}, \sqrt{\varepsilon_{2}\varepsilon_{\pt}\varepsilon_{2\pt}}, \sqrt{ \varepsilon_{2\po\pt}}   \rangle.$$
			 	\end{enumerate}

 			\item  Assume that  $(n_1,n_2,n_3, n_4)=( 1, 1, 1,  1)$. Then the unit group of $\KK$ is :
 			$$E_{\KK}=\langle -1, \varepsilon_{2}, \varepsilon_{\po},\varepsilon_{\pt}, \sqrt{\varepsilon_{2\po}}, \sqrt{\varepsilon_{2\pt}},  \sqrt{\varepsilon_{\po\pt}}, 
 			\sqrt{\varepsilon_{2}^{\alpha a} \varepsilon_{\po}^{\alpha b}\varepsilon_{\pt}^{\alpha c} \sqrt{\eta}^{1+\gamma}} \rangle.$$ 
 			for some  $\alpha $, $\gamma$, $a$, $b$,   and $c$ in $\{0,1\}$ with  $\alpha \not= \gamma$. Where $\eta= \varepsilon_{2\po}\varepsilon_{2\pt}\varepsilon_{2\po\pt}$, $ \varepsilon_{2\po}\varepsilon_{\po\pt}\varepsilon_{2\po\pt}$ or $ \varepsilon_{2\pt}\varepsilon_{\po\pt}\varepsilon_{2\po\pt}$ according to whether $(x\pm1)$ is a square in $\NN$ or $($$(x\pm1)$ is 
 			not a square in $\NN$ and $2\pt(x\pm1)$ is square in $\NN$$)$ or $((x\pm1)$  and $2\pt(x\pm1)$ are not squares in $\NN)$. 		\end{enumerate}
 	\end{theorem}
 	\begin{proof}We use the same techniques as  in the proof of Theorem \ref{MT1A}. Let us start with the first item.
 		\begin{enumerate}[\rm $1)$]
 			\item Assume that $(n_1,n_2,n_3, n_4)=( -1, -1, -1,  1)$, and let $\varepsilon_{2\po\pt}=x+y\sqrt{2\po\pt} $ for some integers $x$ and $y$. 
 			Note that  a F.S.U of $k_i$ is $\{\varepsilon_{2},\varepsilon_{p_i}, \sqrt{\varepsilon_{2}\varepsilon_{p_i}\varepsilon_{2p_i}}\}$, for $i \in  \{1,2\}$ (cf. Lemma \ref{fork1k2}) and a F.S.U of     $k_3$ is $\{\varepsilon_{2},\varepsilon_{\po\pt}, \sqrt{\varepsilon_{2\po\pt}}\}$ or $\{\varepsilon_{2},\varepsilon_{\po\pt},  {\varepsilon_{2\po\pt}}\}$ according to whether $x\pm1$ is a square in $\NN$ or not (cf. Lemma \ref{lmunit}). Thus,  we   distinguish two cases:
 			\begin{enumerate}[ $\bullet$]
 				\item If $x\pm1$ is a square in $\NN$, then:
 				$$E_{k_1}E_{k_2}E_{k_3}=\langle-1,  \varepsilon_{2}, \varepsilon_{\po},\varepsilon_{\pt}, \varepsilon_{\po\pt} ,\sqrt{\varepsilon_{2\po\pt}},  \sqrt{\varepsilon_{2}\varepsilon_{\po}\varepsilon_{2\po}},\sqrt{\varepsilon_{2}\varepsilon_{\pt}\varepsilon_{2\pt}} \rangle.$$
 				So we consider
 				\begin{eqnarray*} 
 					\chi^2=\varepsilon_{2}^a\varepsilon_{\po}^b \varepsilon_{\pt}^c\varepsilon_{\po\pt}^d \sqrt{\varepsilon_{2\po\pt}}^e\sqrt{\varepsilon_{2}\varepsilon_{\po}\varepsilon_{2\po}}^f\sqrt{\varepsilon_{2}\varepsilon_{\pt}\varepsilon_{2\pt}}^g,	 
 				\end{eqnarray*}
 				where $a, b, c, d, e, f$ and $g$ are in $\{0, 1\}$. We have: 
 			 	\begin{eqnarray}
 						N_{\KK/k_2}(\chi^2)&=&
 					\varepsilon_{2}^{2a}(-1)^b \cdot \varepsilon_{\pt}^{2c}\cdot (-1)^d\cdot(-1)^{a_1 e} \cdot  \varepsilon_{2}^{f}\cdot (-1)^{a_2 f}\cdot(\varepsilon_{2}\varepsilon_{\pt}\varepsilon_{2\pt})^{g}\label{16}, \\
 					N_{\KK/k_3}(\chi^2)&=&
 					\varepsilon_{2}^{2a}(-1)^b \cdot (-1)^c\cdot \varepsilon_{\po\pt}^{2d} \cdot  \varepsilon_{2\po\pt}^{e}\cdot \varepsilon_{2}^{g}(-1)^{a_3 g}\label{17}, \\
 					N_{\KK/k_4}(\chi^2)&=&
 					(-1)^a\cdot \varepsilon_{\po}^{2b} \cdot \varepsilon_{\pt}^{2c}\cdot \varepsilon_{\po\pt}^{2d} \cdot  (-1)^{a_4 e}, \\
 					N_{\KK/k_6}(\chi^2)&=&
 					(-1)^a\cdot \varepsilon_{\po}^{2b} \cdot 	(-1)^c\cdot 	(-1)^d \cdot \varepsilon_{2\po\pt}^{e} \cdot  (-1)^{a_5 e}\label{19}, 	                 
 				\end{eqnarray}

for some $a_i\in \{0,1\}$. Thus, we have the following congruent systems : 
	\[
\left \{
\begin{array}{ccc}
	b+d+a_1 e+ a_2 f & \equiv& 0 \pmod 2, \\
	b+c+a_3 g & \equiv& 0 \pmod 2, \\
	a+a_4 e & \equiv& 0 \pmod 2, \\
	a+c+d+a_5 e & \equiv& 0 \pmod 2. \\
\end{array}
\right.
\]
 Note that  $\varepsilon_{2}$ is not a square in $k_2$ and $k_3$, then  \eqref{16} and \eqref{17} give $f=g=0$. As $x\pm1$ is a square in $\NN$, then   $\varepsilon_{2\po\pt}$ is not a square in $k_6$ $($cf. Remark \ref{twosquare}$)$. So    \eqref{19} gives $e=0$ and  from the above system  we deduce that 
 				  $b=c=d$ and $ a= 0$. Therefore,
 				$$\chi^2=(\varepsilon_{\po}\varepsilon_{\pt}\varepsilon_{\po\pt})^b .$$
 				As according to  Lemma \ref{squaretest} we have $\sqrt{\varepsilon_{\po}\varepsilon_{\pt}\varepsilon_{\po\pt}}\in \KK$, then  
 				$$E_{\KK}=\langle -1,  \varepsilon_{2}, \varepsilon_{\po},   \varepsilon_{\pt},\sqrt{\varepsilon_{2\po\pt}}, \sqrt{\varepsilon_{2}\varepsilon_{\po}\varepsilon_{2\po}}, 
 				\sqrt{\varepsilon_{2}\varepsilon_{\pt} \varepsilon_{2\pt}}, \sqrt{\varepsilon_{\po}\varepsilon_{\pt}\varepsilon_{\po\pt}}  \rangle.$$
 				\item If $x\pm1$ is not a square in $\NN$,   then 
 				$$E_{k_1}E_{k_2}E_{k_3}=\langle-1,  \varepsilon_{2}, \varepsilon_{\po},\varepsilon_{\pt}, \varepsilon_{\po\pt} ,\varepsilon_{2\po\pt},  \sqrt{\varepsilon_{2}\varepsilon_{\po}\varepsilon_{2\po}},\sqrt{\varepsilon_{2}\varepsilon_{\pt}\varepsilon_{2\pt}} \rangle.$$
 			Notice that by Remark \ref{twosquare} $\varepsilon_{2\po\pt}$ is a square in $\KK$. Thus, we consider 
 				\begin{eqnarray} 
 					\chi^2=\varepsilon_{2}^a\varepsilon_{\po}^b \varepsilon_{\pt}^c\varepsilon_{\po\pt}^d \sqrt{\varepsilon_{2}\varepsilon_{\po}\varepsilon_{2\po}}^e\sqrt{\varepsilon_{2}\varepsilon_{\pt}\varepsilon_{2\pt}}^f,	 
 				\end{eqnarray}
 				where $a, b, c, d, e,$ and $f$ are in $\{0, 1\}$. 
 					By applying the   maps $N_{\KK/k_2}, N_{\KK/k_3}$ and $N_{\KK/k_4}$, as above, we eliminate   the forms of $\chi^2$  except the following:
 				\begin{eqnarray} 
 					\chi^2= (\varepsilon_{\po} \varepsilon_{\pt}\varepsilon_{\po\pt})^b.	 
 				\end{eqnarray}
 				Note that    $\sqrt{\varepsilon_{\po}\varepsilon_{\pt}\varepsilon_{\po\pt}}\in \KK $ (cf. Lemma  \ref{squaretest}), so we have:  
 				$$E_{\KK}=\langle -1,  \varepsilon_{2}, \varepsilon_{\po},   \varepsilon_{\pt},\sqrt{\varepsilon_{2\po\pt}}, \sqrt{\varepsilon_{2}\varepsilon_{\po}\varepsilon_{2\po}}, 
 				\sqrt{\varepsilon_{2}\varepsilon_{\pt} \varepsilon_{2\pt}}, \sqrt{\varepsilon_{\po}\varepsilon_{\pt}\varepsilon_{\po\pt}}  \rangle.$$
 			\end{enumerate}
 			
 			\item Now assume that $(n_1,n_2,n_3, n_4)=(-1, 1,-1,1)$. By Lemmas \ref{fork1k2} and \ref{lmunit}, we have:  
 			\begin{enumerate}[$\star$]
 				\item A F.S.U of $k_1$ is given by $\{\varepsilon_{2},\varepsilon_{\po}, \sqrt{\varepsilon_{2}\varepsilon_{\po}\varepsilon_{2\po}}\}$,
 				\item A F.S.U of $k_2$ is given by $\{\varepsilon_{2},\varepsilon_{\pt}, \sqrt{\varepsilon_{2\pt}}\}$,
 				\item A F.S.U of $k_3$ is given by $\{\varepsilon_{2},\varepsilon_{\po\pt}, \sqrt{\varepsilon_{2\po\pt}}\}$ or $\{\varepsilon_{2},\varepsilon_{\po\pt},\varepsilon_{2\po\pt}\}$ according to whether $x\pm1$ is a square in $\NN$ or not.  
 			\end{enumerate} 
 			\begin{enumerate}[$\bullet$]
 				\item If  $x\pm1$ is a square in $\NN$, we have  
 				$$E_{k_1}E_{k_2}E_{k_3}=\langle-1,  \varepsilon_{2}, \varepsilon_{\po},\varepsilon_{\pt}, \varepsilon_{\po\pt}, \sqrt{\varepsilon_{2\pt}} ,\sqrt{\varepsilon_{2\po\pt}},  \sqrt{\varepsilon_{2}\varepsilon_{\po}\varepsilon_{2\po}} \rangle.$$
 				As before we consider $\chi\in \KK$ such that:
 				\begin{eqnarray*} 
 					\chi^2=\varepsilon_{2}^a\varepsilon_{\po}^b \varepsilon_{\pt}^c\varepsilon_{\po\pt}^d \sqrt{\varepsilon_{2\pt}}^e\sqrt{\varepsilon_{2\po\pt}}^f\sqrt{\varepsilon_{2}\varepsilon_{\po}\varepsilon_{2\po}}^g,
 				\end{eqnarray*}
 				where $a, b, c, d, e, f$ and $g$ are in $\{0, 1\}$. We have:
 				\begin{eqnarray}
 						N_{\KK/k_2}(\chi^2)&=&
 					\varepsilon_{2}^{2a}(-1)^b \cdot \varepsilon_{\pt}^{2c}\cdot (-1)^d\cdot \varepsilon_{2\pt}^e  \cdot (-1)^{a_1 f} \cdot \varepsilon_{2}^{g}\cdot (-1)^{b_1 g} \label{21} \\
 					N_{\KK/k_3}(\chi^2)&=&
 					\varepsilon_{2}^{2a}\cdot(-1)^b \cdot (-1)^c\cdot \varepsilon_{\po\pt}^{2d} \cdot  (-1)^{ue} \varepsilon_{2\po\pt}^{f}\\
 						N_{\KK/k_4}(\chi^2)&=&
 					(-1)^a\cdot \varepsilon_{2\po}^{2b} \cdot \varepsilon_{2\pt}^{2c}\cdot \varepsilon_{\po\pt}^{2d} \cdot  (-1)^{(u+1)e} \cdot (-1)^{a_2 f}\\
 						N_{\KK/k_6}(\chi^2)&=&
 					(-1)^a\cdot \varepsilon_{\po}^{2b} \cdot (-1)^c\cdot (-1)^d \cdot (-1)^e \varepsilon_{2\pt}^e \cdot \varepsilon_{2\po\pt}^f  \cdot(-1)^{a_3 f}\label{24} \\
 					N_{\KK/k_5}(\chi^2)&=&
 					(-1)^a\cdot (-1)^b\cdot \varepsilon_{2\pt}^{2c} \cdot (-1)^d\cdot (-1)^{(u+1)e} \cdot \varepsilon_{2\po\pt}^e \cdot  (-1)^{a_4 e} \label{25}               
 				\end{eqnarray}
 				where $a_i$ and $b_i$ are in $\{0,1\}$ and $u$ is defined
 				as in Lemma   \ref{calcul} for $p=\pt$. As $\varepsilon_{2}$ is not a square in $k_2$, then   \eqref{21} gives  $g=0$ and from the fact that $\varepsilon_{2\po\pt}$ is not a square in $k_5$, we have        $e=0$, see \eqref{25}. Moreover $\varepsilon_{2\pt}$ and $\varepsilon_{2\po\pt}$ are not squares in $k_6$ then from \eqref{24} we have $e=f.$ Therefore, we extract the following system. 
 				 	\[
 				\left \{
 				\begin{array}{ccl}
 					b+d & \equiv& 0 \pmod 2 \\
 					b+c & \equiv& 0 \pmod 2 \\
 					a & =& 0  \\
 					a+c+d & \equiv& 0 \pmod 2 \\
 					a+b+d & \equiv& 0 \pmod 2 \\
 					
 				\end{array}
 				\right.
 				\]
 				From this, we eliminate all the forms of $\chi^2$ except the following:
 			  $$\chi^2=(\varepsilon_{\po}\varepsilon_{\pt}\varepsilon_{\po\pt} )^b,	 $$
 			 Hence, from the fact that $\varepsilon_{\po}\varepsilon_{\pt}\varepsilon_{\po\pt}$ is a square in $\KK$ (cf. Lemma \ref{squaretest}),   we derive that
 				$$E_{\KK}=\langle-1,  \varepsilon_{2}, \varepsilon_{\po},\varepsilon_{\pt}, \sqrt{\varepsilon_{2\pt}}, \sqrt{\varepsilon_{2\po\pt}} ,  \sqrt{\varepsilon_{2}\varepsilon_{\po}\varepsilon_{2\po}}, \sqrt{\varepsilon_{\po}\varepsilon_{\pt}\varepsilon_{\po\pt}} \rangle.$$
 				\item Similarly if  $x\pm1$ is not a square in $\NN$, we have 
 				$$E_{k_1}E_{k_2}E_{k_3}=\langle-1,  \varepsilon_{2}, \varepsilon_{\po},\varepsilon_{\pt}, \varepsilon_{\po\pt},\varepsilon_{2\po\pt}, \sqrt{\varepsilon_{2\pt}},   \sqrt{\varepsilon_{2}\varepsilon_{\po}\varepsilon_{2\po}} \rangle.$$
 				With analogous procedure  as above, we check that we have the same result in this case as well.
 			\end{enumerate}
 			
 			\item Assume that $(n_1,n_2,n_3, n_4)=( 1, 1, -1, 1)$. By Lemmas \ref{fork1k2} and \ref{lmunit} the F.S.U of $k_i$ is \begin{enumerate}[$\star$]
 				\item $\{\varepsilon_{2},\varepsilon_{p_i}, \sqrt{\varepsilon_{2p_i}}\}$ for $i=1,2$.
 				\item $\{\varepsilon_{2},\varepsilon_{\po\pt}, \sqrt{\varepsilon_{2\po\pt}}\}$ or $\{\varepsilon_{2},\varepsilon_{\po\pt},\varepsilon_{2\po\pt}\}$ according to whether $x\pm1$ is a square in $\NN$ or not, for $i=3$.  
 			\end{enumerate} 
 			\begin{enumerate}[$\bullet$]
 				
 				\item If  $x\pm1$ is not a square in $\NN$, then \label{cassimil}
 				$$E_{k_1}E_{k_2}E_{k_3}=\langle-1,  \varepsilon_{2}, \varepsilon_{\po},\varepsilon_{\pt}, \varepsilon_{\po\pt},\sqrt{\varepsilon_{2\po}}, \sqrt{\varepsilon_{2\pt}}, \varepsilon_{2\po\pt} \rangle.$$
 			 			Notice that by Remark \ref{twosquare} $\varepsilon_{2\po\pt}$ is a square in $\KK$.  	So we consider $\chi \in \KK$ such that   \begin{eqnarray}
 					\chi^2=\varepsilon_{2}^a\varepsilon_{\po}^b \varepsilon_{\pt}^c\varepsilon_{\po\pt}^d\sqrt{\varepsilon_{2\po}}^e \sqrt{\varepsilon_{2\pt}}^f,	 
 				\end{eqnarray}
 				where $a, b, c, d, e$ and $f$ are in $\{0, 1\}$.
 				By applying the six norm maps we get 
 				\begin{eqnarray}
 					N_{\KK/k_2}(\chi^2)&=&
 					\varepsilon_2^{2a}\cdot (-1)^b \cdot\varepsilon_{\pt}^{2c}(-1)^d\cdot (-1)^{ue} \cdot  \varepsilon_{2\pt}^f\nonumber \\
 					N_{\KK/k_3}(\chi^2)&=&
 					\varepsilon_2^{2a}\cdot (-1)^b \cdot(-1)^c \cdot \varepsilon_{\po\pt}^{2d}\cdot (-1)^{ue} \cdot  (-1)^{vf} \nonumber\\
 					N_{\KK/k_4}(\chi^2)&=&
 					(-1)^a\cdot \varepsilon_{\po}^{2b} \cdot\varepsilon_{\pt}^{2c} \cdot \varepsilon_{\po\pt}^{2d}\cdot (-1)^{(u+1)e} \cdot  (-1)^{(v+1)f}\nonumber\\
 					N_{\KK/k_6}(\chi^2)&=&
 					(-1)^a\cdot \varepsilon_{\po}^{2b} \cdot(-1)^{c} \cdot(-1)^{d}\cdot (-1)^{(u+1)e} \cdot  (-1)^{f}\cdot \varepsilon_{2\pt}^{f}\label{aa}\\
 					N_{\KK/k_1}(\chi^2)&=&
 					\varepsilon_{2}^{2a}\cdot \varepsilon_{\po}^{2b} \cdot(-1)^{c} \cdot(-1)^{d}\cdot \varepsilon_{2\po}^{e} \cdot  (-1)^{vf}\nonumber\\
 					N_{\KK/k_5}(\chi^2)&=& (-1)^{a}\cdot(-1)^{b}\cdot\varepsilon_{\pt}^{2c}
 					\cdot(-1)^d\cdot (-1)^e\varepsilon_{2\po}^{e} \cdot(-1)^{(v+1)e}\label{bb}                 
 				\end{eqnarray}
 			 	where $u$ and $v$ are defined in Lemma \ref{calcul} for $p=\po$ and $p=\pt$ respectively. Therefore 
 			 	\[
 				\left \{
 				\begin{array}{ccc}
 					b+d+ue & \equiv& 0 \pmod 2 \\
 					b+c+ue+vf & \equiv& 0 \pmod 2 \\
 					a+(u+1)e+(v+1)f & \equiv& 0 \pmod 2 \\
 					a+c+d+(u+1)e+f & \equiv& 0 \pmod 2 \\
 					c+d+vf & \equiv& 0 \pmod 2 \\
 					a+b+d+e+(v+1)f & \equiv& 0 \pmod 2 \\
 					
 				\end{array}
 				\right.
 				\]
 		 	Notice that $\sqrt{\varepsilon_{2\po}}\notin k_5$, and  $\sqrt{\varepsilon_{2\pt}}\notin k_6$. It follows from   \eqref{aa} and \eqref{bb}, that  we have $e=f=0$. So  the above system implies that $a=0$ and $b=c=d$. Thus, all the forms of $\chi^2$ are eliminated except the following: 
 				$$\chi^2=(\varepsilon_{\po}\varepsilon_{\pt}\varepsilon_{\po\pt} )^b.	 $$
 				As   according to Lemma \ref{squaretest} $\varepsilon_{\po}\varepsilon_{\pt}\varepsilon_{\po\pt}$ is a square in $\KK$, then 
 				$$E_{\KK}=\langle-1,  \varepsilon_{2}, \varepsilon_{\po},\varepsilon_{\pt}, \sqrt{\varepsilon_{2\po}},\sqrt{\varepsilon_{2\pt}}, \sqrt{\varepsilon_{2\po\pt}},\sqrt{\varepsilon_{\po}\varepsilon_{\pt}\varepsilon_{\po\pt}} \rangle.$$
 				
 				\item If $x\pm1$ is a square in $\NN$, we have  
 			 	$$E_{k_1}E_{k_2}E_{k_3}=\langle-1,  \varepsilon_{2}, \varepsilon_{\po},\varepsilon_{\pt}, \varepsilon_{\po\pt},\sqrt{\varepsilon_{2\po}}, \sqrt{\varepsilon_{2\pt}} ,\sqrt{\varepsilon_{2\po\pt}}   \rangle.$$ 
 			 	 Let   $\chi$ be an element of $\KK  $  such that  
 				then we have \begin{eqnarray}\label{chi17a}
 					\chi^2=\varepsilon_{2}^a\varepsilon_{\po}^b \varepsilon_{\pt}^c\varepsilon_{\po\pt}^d\sqrt{\varepsilon_{2\po}}^e \sqrt{\varepsilon_{2\pt}}^f\sqrt{\varepsilon_{2\po\pt}}^g,	 
 				\end{eqnarray}
 				where $a, b, c, d, e, f$ and $g$ are in $\{0, 1\}$.
 				By applying the norm maps $N_{\KK/k_6},N_{\KK/k_5}$ and $N_{\KK/k_2}$ we get : 
 				 	\begin{eqnarray}
 						N_{\KK/k_6}(\chi^2)&=&
 						(-1)^a\cdot \varepsilon_{\po}^{2b} \cdot (-1)^c\cdot (-1)^d \cdot (-1)^{(u+1)e} \cdot (-1)^f \cdot\varepsilon_{2\pt}^f \cdot \varepsilon_{2\po\pt}^g  \cdot(-1)^{a_1 g},\nonumber \\
 						N_{\KK/k_5}(\chi^2)&=&
 						(-1)^a\cdot (-1)^b\cdot \varepsilon_{\pt}^{2c} \cdot (-1)^d\cdot (-1)^{e} \cdot \varepsilon_{2\po}^e\cdot (-1)^{(v+1)f}\cdot\varepsilon_{2\po\pt}^f \cdot  (-1)^{a_2 f}, \nonumber  \\              
 						N_{\KK/k_2}(\chi^2)&=&
 					\varepsilon_{2}^{2a}(-1)^b \cdot \varepsilon_{\pt}^{2c}\cdot (-1)^d\cdot (-1)^{ue}\varepsilon_{2\pt}^e  \cdot (-1)^{a_3 e}, \nonumber 
 				\end{eqnarray}
 				
 					with $a_i\in \{0,1\}$ and  $u,v$  are defined in Lemma \ref{calcul}   for $p=\po$ and $p=\pt$ respectively.  As $\varepsilon_{2\pt}$ and $\varepsilon_{2\po\pt}$ are not squares in $k_6$ the from the first equality we have $f=g.$
 				 Also $\varepsilon_{2\po\pt}$ and $\varepsilon_{2\po}$ are not squares in $k_5$ then from the second equality we get $e=f$. Thus, we extract the following system : 
 			 	\[
 				\left \{
 				\begin{array}{ccc}
 					b+d+ue+a_3 e &\equiv& 0\pmod2, \\
 					a+c+d+(u+a_1) e&\equiv& 0\pmod2, \\
 					a+b+d+(v+a_2)e&\equiv& 0\pmod2. \\
 				 	\end{array}
 				\right.
 				\]
   So if $e=0$ we get from the above congruence  equations that $b=d=c$ and $a=0$ which gives $\chi^2=(\varepsilon_{\po} \varepsilon_{\pt}\varepsilon_{\po\pt})^b$, which admit a solution in $\KK$  according to Lemma \ref{squaretest} . Now if $e=1$, then
 				\begin{eqnarray}
 					\chi^2=\varepsilon_{2}^a\varepsilon_{\po}^b \varepsilon_{\pt}^c\varepsilon_{\po\pt}^d\sqrt{\varepsilon_{2\po}} \sqrt{\varepsilon_{2\pt}}\sqrt{\varepsilon_{2\po\pt}},	 
 				\end{eqnarray}
 				where   
 			 	\[
 			\left \{
 			\begin{array}{ccc}
 				b+d+u+a_3  &\equiv& 0\pmod2, \\
 				a+c+d+u+a_1 &\equiv& 0\pmod2, \\
 				a+b+d+v+a_2&\equiv& 0\pmod2. \\
 			 	\end{array}
 			\right.
 			\]
 		 From this we deduce the second part of the third item.

 			\end{enumerate}
 			\item Assume that $(n_1,n_2,n_3, n_4)=(-1, 1, 1,-1)$. By Lemmas \ref{fork1k2} and \ref{lmunit}, we have:
 			\begin{enumerate}[$\star$]
 				\item $\{\varepsilon_{2},\varepsilon_{\po}, \sqrt{\varepsilon_{2}\varepsilon_{\po}\varepsilon_{2\po}}\}$ is a FSU of $k_1$,
 				\item $\{\varepsilon_{2},\varepsilon_{\pt}, \sqrt{\varepsilon_{2\pt}}\}$ is a FSU of $k_2$,
 				\item $\{\varepsilon_{2},\varepsilon_{\po\pt}, \varepsilon_{2\po\pt}\}$ is a FSU of $k_3$.  
 			\end{enumerate} 
 			Then $$E_{k_1}E_{k_2}E_{k_3}=\langle-1,  \varepsilon_{2}, \varepsilon_{\po},\varepsilon_{\pt}, \varepsilon_{\po\pt},\varepsilon_{2\po\pt},\sqrt{\varepsilon_{2}\varepsilon_{\po}\varepsilon_{2\po}},\sqrt{\varepsilon_{2\pt}}    \rangle.$$
 			Note that by Remark \ref{twosquare} $\varepsilon_{\po\pt}$ is a square in $\KK$.  	So we consider $\chi \in \KK$ such that 
 			$$\chi^2=\varepsilon_{2}^a\varepsilon_{\po}^b \varepsilon_{\pt}^c\varepsilon_{2\po\pt}^d \sqrt{\varepsilon_{2\pt}}^e\sqrt{\varepsilon_{2}\varepsilon_{\po}\varepsilon_{2\po}}^f$$
 			where $a, b, c, d, e$ and $f$ are in $\{0, 1\}$. 
 		Proceeding similarly as above  by applying the norm maps $N_{\KK/k_2},N_{\KK/k_3}, N_{\KK/k_4}$ and $N_{\KK/k_5}$, we check that $e=f=0$ and $a=b=c=d$.
 This eliminates all the forms of $\chi^2$ except the following:  
 			$$\chi^2=(\varepsilon_{2}\varepsilon_{\po}\varepsilon_{\pt}\varepsilon_{2\po\pt})^a .$$
 			By Lemma \ref{squaretest}, $\varepsilon_{2}\varepsilon_{\po}\varepsilon_{\pt}\varepsilon_{2\po\pt}$ is a square in $\KK$.   Therefore, 
 			$$E_{\KK}=\langle-1,  \varepsilon_{2}, \varepsilon_{\po},\varepsilon_{\pt}, \sqrt{\varepsilon_{\po\pt}},\sqrt{\varepsilon_{2}\varepsilon_{\po}\varepsilon_{2\po}},\sqrt{\varepsilon_{2\pt}}, \sqrt{\varepsilon_{2}\varepsilon_{\po}\varepsilon_{\pt}\varepsilon_{2\po\pt}}    \rangle.$$

 			\item Assume that $(n_1,n_2,n_3, n_4)=(-1,-1, 1,-1)$. So   by Lemmas \ref{fork1k2} and \ref{lmunit} a F.S.U of $k_i$ is \begin{enumerate}[$\star$]
 				\item $\{\varepsilon_{2},\varepsilon_{p_i}, \sqrt{\varepsilon_{2}\varepsilon_{p_i}\varepsilon_{2p_i}}\}$ for $i=1,2$.
 				\item $\{\varepsilon_{2},\varepsilon_{\po\pt}, \varepsilon_{2\po\pt}\}$ for $i=3$.  
 			\end{enumerate} 
 			Then 
 			$$E_{k_1}E_{k_2}E_{k_3}=\langle-1,  \varepsilon_{2}, \varepsilon_{\po},\varepsilon_{\pt}, \varepsilon_{\po\pt},\varepsilon_{2\po\pt},\sqrt{\varepsilon_{2}\varepsilon_{\po}\varepsilon_{2\po}},\sqrt{\varepsilon_{2}\varepsilon_{\pt}\varepsilon_{2\pt}}    \rangle.$$
 			By Remark \ref{twosquare} $\varepsilon_{\po\pt}$ is a square in $\KK$.  	So we consider $\chi \in \KK$ such that 
 			$$\chi^2=\varepsilon_{2}^a\varepsilon_{\po}^b \varepsilon_{\pt}^c\varepsilon_{2\po\pt}^d \sqrt{\varepsilon_{2}\varepsilon_{\po}\varepsilon_{2\po}}^e\sqrt{\varepsilon_{2}\varepsilon_{\pt}\varepsilon_{2\pt}}^f$$
 			where $a, b, c, d, e$ and $f$ are in $\{0, 1\}$.   
 			 	By applying the norm maps $N_{\KK/k_2},N_{\KK/k_3}$ and  $N_{\KK/k_4}$ we check  that  $e=f=0$ and $a=b=c=d$.
 			Thus eliminates all the forms of $\chi^2$ except the following 
 				$$\chi^2=\varepsilon_{2}^a\varepsilon_{\po}^a \varepsilon_{\pt}^a\varepsilon_{2\po\pt}^a.  $$
 			  As according to Lemma \ref{squaretest} $\sqrt{\varepsilon_{2}\varepsilon_{\po}\varepsilon_{\pt}\varepsilon_{2\po\pt}}$, we deduce that
 		 	$$E_{\KK}=\langle-1,  \varepsilon_{2}, \varepsilon_{\po},\varepsilon_{\pt}, \sqrt{\varepsilon_{\po\pt}},\sqrt{\varepsilon_{2}\varepsilon_{\po}\varepsilon_{2\po}},\sqrt{\varepsilon_{2}\varepsilon_{\pt}\varepsilon_{2\pt}}, \sqrt{\varepsilon_{2}\varepsilon_{\po}\varepsilon_{\pt}\varepsilon_{2\po\pt}}    \rangle.$$

 			\item Assume that $(n_1,n_2,n_3, n_4)=( 1, 1, 1, -1)$. By Lemmas \ref{fork1k2} and \ref{lmunit} the F.S.U of $k_i$ is \begin{enumerate}[$\star$]
 				\item $\{\varepsilon_{2},\varepsilon_{p_i}, \sqrt{\varepsilon_{2p_i}}\}$ for $i=1,2$.
 				\item $\{\varepsilon_{2},\varepsilon_{\po\pt}, \varepsilon_{2\po\pt} \}$ for $i=3$.  
 			\end{enumerate} 
 			So we have 
 			$$E_{k_1}E_{k_2}E_{k_3}=\langle-1,  \varepsilon_{2}, \varepsilon_{\po},\varepsilon_{\pt}, \varepsilon_{\po\pt},\varepsilon_{2\po\pt},\sqrt{\varepsilon_{2\po}}, \sqrt{\varepsilon_{2\pt}}   \rangle.$$
 			Noting that  $\sqrt{\varepsilon_{\po\pt}}$   and $\sqrt{\varepsilon_{2}\varepsilon_{\po}\varepsilon_{\pt}\varepsilon_{2\po\pt}}$ are  in $\KK$ (cf. Remark \ref{twosquare} and Lemma \ref{squaretest}), then with the same procedure as in proof of the first case of the third item gives:
 			$$E_{\KK}=\langle-1,  \varepsilon_{2}, \varepsilon_{\po},\varepsilon_{\pt}, \sqrt{\varepsilon_{\po\pt}},\sqrt{\varepsilon_{2\po}}, \sqrt{\varepsilon_{2\pt}}, \sqrt{\varepsilon_{2}\varepsilon_{\po}\varepsilon_{\pt}\varepsilon_{2\po\pt}}   \rangle.$$

 			\item Now assume that $(n_1,n_2,n_3, n_4)=(-1, 1, 1, 1)$. So a F.S.U of $k_i$ is 
 			\begin{enumerate}[$\star$]
 				\item $\{\varepsilon_{2},\varepsilon_{\po}, \sqrt{\varepsilon_{2}\varepsilon_{\po}\varepsilon_{2\po}}\}$ for $i=1$.
 				\item $\{\varepsilon_{2},\varepsilon_{\pt}, \sqrt{\varepsilon_{2\pt}}\}$ for $i=2$.
 				\item A F.S.U of $k_3$ is given by $\{\varepsilon_{2},\varepsilon_{\po\pt}, \sqrt{\varepsilon_{2\po\pt}}\}$ or $\{\varepsilon_{2},\varepsilon_{\po\pt},\sqrt{\varepsilon_{\po\pt}\varepsilon_{2\po\pt}}\}$ according to whether $x\pm1$ is a square in $\NN$ or not.
 				
 			\end{enumerate}
 			\begin{enumerate}[$\bullet$]
 				\item If $x\pm 1$ is a square in $\NN$, then
 			 $$E_{k_1}E_{k_2}E_{k_3}=\langle-1,  \varepsilon_{2}, \varepsilon_{\po},\varepsilon_{\pt}, \varepsilon_{\po\pt},\sqrt{\varepsilon_{2\pt}},\sqrt{\varepsilon_{2\po\pt}},\sqrt{\varepsilon_{2}\varepsilon_{\po}\varepsilon_{2\po}}    \rangle.$$
 			So let us consider $\chi\in \KK$ such that 
 			$$\chi^2=\varepsilon_{2}^a\varepsilon_{\po}^b \varepsilon_{\pt}^c\sqrt{\varepsilon_{2\pt}}^d \sqrt{\varepsilon_{2\po\pt}}^e\sqrt{\varepsilon_{2}\varepsilon_{\po}\varepsilon_{2\po}}^f,$$
 			where $a, b, c, d, e$ and $f$ are in $\{0, 1\}$. Thus, we have: 
 			\begin{eqnarray}
 				N_{\KK/k_2}(\chi^2)&=&
 				\varepsilon_2^{2a}\cdot (-1)^b \cdot\varepsilon_{\pt}^{2c}\cdot\varepsilon_{2\pt}^d \cdot (-1)^{a_1 e} \cdot  \varepsilon_{2}^f (-1)^{b_1 f}\label{n1},\\
 				N_{\KK/k_3}(\chi^2)&=&
 				\varepsilon_2^{2a}\cdot (-1)^b \cdot(-1)^c \cdot  (-1)^{du} \cdot \varepsilon_{2\po\pt}^e,\nonumber \\
 				N_{\KK/k_4}(\chi^2)&=&
 				(-1)^a\cdot \varepsilon_{\po}^{2b} \cdot\varepsilon_{\pt}^{2c} \cdot  (-1)^{(u+1)d} \cdot  (-1)^{a_2 e},\nonumber\\
 				N_{\KK/k_6}(\chi^2)&=&
 				(-1)^a\cdot \varepsilon_{\po}^{2b} \cdot(-1)^{c} \cdot(-1)^{d}\cdot \varepsilon_{2\pt}^{d}\cdot(-1)^{a_3 e} \cdot  \varepsilon_{2\po\pt}^{e},\label{..}\\
 				N_{\KK/k_1}(\chi^2)&=&
 				\varepsilon_{2}^{2a}\cdot \varepsilon_{\po}^{2b} \cdot(-1)^{c} \cdot(-1)^{du}\cdot   (-1)^{a_4 e},\nonumber\\
 				N_{\KK/k_5}(\chi^2)&=& (-1)^{a}\cdot(-1)^{b}\cdot\varepsilon_{\pt}^{2c}
 				\cdot(-1)^{(u+1)d}\cdot \varepsilon_{2\po\pt}^{e} \cdot(-1)^{a_5 e},\label{n6}                 
 			\end{eqnarray}
 		 	where $u$ is given by Lemma \ref{calcul} for $p=p_2$ and $a_i \in \{0,1\}$. Therefore 
 		 	\[
 			\left \{
 			\begin{array}{ccc}
 				b+a_1 e +b_1 f & \equiv& 0 \pmod 2, \\
 				b+c+du & \equiv& 0 \pmod 2 \\
 				a+(u+1)d +a_2 e & \equiv& 0 \pmod 2, \\
 				a+c+d+a_3 e & \equiv& 0 \pmod 2, \\
 				c+du+a_4 e & \equiv& 0 \pmod 2, \\
 				a+b+d(u+1)+a_5 e & \equiv& 0 \pmod 2. \\
 				
 			\end{array}
 			\right.
 			\]
 			As $\varepsilon_2$ is not a square in $k_2$ we get from \eqref{n1} $f=0$, and from \eqref{n6} we get $e=0$, in fact $\varepsilon_{2\po\pt}$ is not a square in $k_5$. Then the above system of congruence 
 			equations implies that $a=b=c=d=0$.  Notice that $\sqrt{\varepsilon_{\po\pt}}$ is a square in $\KK$. Thus,
 			$$E_{\KK}=\langle-1,  \varepsilon_{2}, \varepsilon_{\po},\varepsilon_{\pt},\sqrt{\varepsilon_{\po\pt}},  \sqrt{\varepsilon_{2\po\pt}} ,\sqrt{\varepsilon_{2\pt}} , \sqrt{\varepsilon_{2}\varepsilon_{\po}\varepsilon_{2\po}} \rangle.$$ 
 			
 			\item If $x\pm 1$ is \textbf{not} a square in $\NN$, then 
 			$$E_{k_1}E_{k_2}E_{k_3}=\langle-1,  \varepsilon_{2}, \varepsilon_{\po},\varepsilon_{\pt},\varepsilon_{\po\pt},\sqrt{\varepsilon_{2}\varepsilon_{\po}\varepsilon_{2\po}},\sqrt{\varepsilon_{2 \pt}}, \sqrt{\varepsilon_{\po\pt}\varepsilon_{2\po\pt}}   \rangle.$$
 		 	Notice that     $\sqrt{\varepsilon_{\po\pt}}\in \KK$ $($cf. Remark \ref{twosquare}$)$. Thus,  we consider $\chi\in \KK$ such that  
 			$$\chi^2=\varepsilon_{2}^a\varepsilon_{\po}^b \varepsilon_{\pt}^c\sqrt{\varepsilon_{2}\varepsilon_{\po}\varepsilon_{2\po}}^d \sqrt{\varepsilon_{2\pt}}^e\sqrt{\varepsilon_{\po\pt}\varepsilon_{2\po\pt}}^f,$$
 			where $a, b, c, d, e$ and $f$ are in $\{0, 1\}$. 
 		Let	 $u$ be as in   Lemma \ref{calcul} for $p=p_2$.
 			Therefore, we have: 
 			\begin{eqnarray}
 				N_{\KK/k_2}(\chi^2)&=&
 				\varepsilon_2^{2a}\cdot (-1)^b \cdot\varepsilon_{\pt}^{2c}\cdot\varepsilon_{2}^{d} \cdot (-1)^{a_1 d} \cdot  \varepsilon_{2\pt}^e (-1)^{a_2 f},\label{first}\\
 				N_{\KK/k_3}(\chi^2)&=&
 				\varepsilon_2^{2a}\cdot (-1)^b \cdot(-1)^c \cdot  (-1)^{ue} \cdot (\varepsilon_{\po\pt}\varepsilon_{2\po\pt})^f,\nonumber \\
 				N_{\KK/k_4}(\chi^2)&=&
 				(-1)^a\cdot \varepsilon_{\po}^{2b} \cdot\varepsilon_{\pt}^{2c} \cdot  (-1)^{(u+1)e}\cdot \varepsilon_{  \po\pt}^f \cdot  (-1)^{a_3 f},\nonumber\\
 				N_{\KK/k_6}(\chi^2)&=&
 				(-1)^a\cdot \varepsilon_{\po}^{2b} \cdot(-1)^{c} \cdot(-1)^{e}\cdot \varepsilon_{2\pt}^{e}\cdot \varepsilon_{  2\po\pt}^f \cdot(-1)^{a_4 f}\label{five}, \\
 				N_{\KK/k_1}(\chi^2)&=&
 				\varepsilon_{2}^{2a}\cdot \varepsilon_{\po}^{2b} \cdot(-1)^{c} \cdot(-1)^{ue}\cdot   (-1)^{a_5 f},\nonumber\\
 				N_{\KK/k_5}(\chi^2)&=& (-1)^{a}\cdot(-1)^{b}\cdot\varepsilon_{\pt}^{2c}
 				\cdot(-1)^{(u+1)e}\cdot \varepsilon_{2\po\pt}^{f} \cdot(-1)^{a_6 e},\label{six}                 
 			\end{eqnarray}
 	 where $a_i \in \{0,1\}$.	As $\varepsilon_{2}$ is not a square in $k_2$, then Equality \eqref{first} gives $d=0$. Furthermore, we have:
 		 	\[
 			\left \{
 			\begin{array}{ccl}
 				b+a_2 f & \equiv& 0 \pmod 2, \\
 				b+c+ue & \equiv& 0 \pmod 2 ,\\
 				a+(u+1)e +a_3 f & \equiv& 0 \pmod 2, \\
 				a+c+e+a_4 f & \equiv& 0 \pmod 2, \\
 				c+ue+a_5 f & \equiv& 0 \pmod 2, \\
 				a+b+(u+1)e+a_6 f & \equiv& 0 \pmod 2 .\\
 				
 			\end{array}
 			\right.
 			\]
 		As $(x\pm1)$ is not a square in $\NN$, so    either $\sqrt{\varepsilon_{  2\po\pt}}\in k_5$ or $\sqrt{\varepsilon_{  2\po\pt}}\in k_6$ $($cf. Remark \ref{twosquare}$)$.
 		Thus,  one of the equalities \eqref{five} and \eqref{six} implies that $f=0$ and so the above system is equivalent to: 
 			\[
 			\left \{
 			\begin{array}{ccl}
 				b & =& 0 , \\

 				a+c+e & \equiv& 0 \pmod 2, \\
 				c+ue & \equiv& 0 \pmod 2 ,\\
 				a+ (u+1)e & \equiv& 0 \pmod 2.
 				
 			\end{array}
 			\right.
 			\]
 		Furthermore,  we have $\chi^2=\varepsilon_{2}^a\varepsilon_{\po}^b \varepsilon_{\pt}^c \sqrt{\varepsilon_{2\pt}}^e$. 	Now if $e=0$, then we deduce that $a=b=c=0$. If $e=1$, then $c=u=a+1$ and 
 		 	$$\chi^2=\varepsilon_{2}^a \varepsilon_{\pt}^{a+1} \sqrt{\varepsilon_{2\pt}}.$$
 		 It follows that 
 		 	$$E_{\KK}=\langle-1,  \varepsilon_{2}, \varepsilon_{\po},\varepsilon_{\pt}, \sqrt{\varepsilon_{\po\pt}},\sqrt{\varepsilon_{  2\po\pt}},\sqrt{\varepsilon_{2}\varepsilon_{\po}\varepsilon_{2\po}},
 		 	\sqrt{\varepsilon_{2}^{a\alpha}\varepsilon_{\pt}^{(1+a)\alpha}\sqrt{\varepsilon_{2\pt}}^{\gamma+1} }   \rangle.$$ 
 		 	where  $a$, $\alpha$ and  $\gamma$ are elements in $  \{0,1\}$ such that $a \equiv u+1 \pmod 2$ and $\alpha\not=\gamma $. Notice that $\gamma=0$ if and only if 
 			$\varepsilon_{2}^a \varepsilon_{\pt}^{a+1} \sqrt{\varepsilon_{2\pt}}$ is a square in $\KK$.

 		\end{enumerate}
 				\item  Assume now that $(n_1,n_2,n_3, n_4)=(-1,-1,1,1)$. According to Lemmas \ref{fork1k2} and \ref{lmunit} the F.S.U of $k_i$ is  $\{\varepsilon_{2},\varepsilon_{p_i}, \sqrt{\varepsilon_{2}\varepsilon_{p_i}\varepsilon_{2p_i}}\}$ for $i=1,2$. We have:
 			 	\begin{enumerate}[$\bullet$]
 					\item If $x\pm1$ is a square in $\NN$, then $\{\varepsilon_{2},\varepsilon_{\po\pt}, \sqrt{\varepsilon_{2\po\pt}}\}$ is the F.S.U of $k_3$.  
 			 	Then $$E_{k_1}E_{k_2}E_{k_3}=\langle-1,  \varepsilon_{2}, \varepsilon_{\po},\varepsilon_{\pt}, \varepsilon_{\po\pt},\sqrt{\varepsilon_{2\po\pt}},\sqrt{\varepsilon_{2}\varepsilon_{\po}\varepsilon_{2\po}},\sqrt{\varepsilon_{2}\varepsilon_{\pt}\varepsilon_{2\pt}}    \rangle.$$

 			By Remark \ref{twosquare} $\varepsilon_{\po\pt}$ is a square in $\KK$.  	So let $\chi \in \KK$ such that 
 			 \begin{eqnarray}\label{chi17}
 				\chi^2=\varepsilon_{2}^a\varepsilon_{\po}^b \varepsilon_{\pt}^c\sqrt{\varepsilon_{2\po\pt}}^d\sqrt{\varepsilon_{2}\varepsilon_{\po}\varepsilon_{2\po}}^e \sqrt{\varepsilon_{2}\varepsilon_{\pt}\varepsilon_{2\pt}}^f,	 
 			\end{eqnarray}
 			where $a, b, c, d, e$ and $f$ are in $\{0, 1\}$.  
 			Applying norm maps $N_{\KK/k_2},N_{\KK/k_3},N_{\KK/k_4}, N_{\KK/k_5}$ and $N_{\KK/k_6}$, we check that $a=b=c=d=e=f=0$. This means that all forms of $\chi^2$ are eliminated.  
 			Therefore, 
 			$$E_{\KK}=\langle-1,  \varepsilon_{2}, \varepsilon_{\po},\varepsilon_{\pt},\sqrt{\varepsilon_{\po\pt}},  \sqrt{\varepsilon_{2\po\pt}} ,  \sqrt{\varepsilon_{2}\varepsilon_{\po}\varepsilon_{2\po}}, \sqrt{\varepsilon_{2}\varepsilon_{\pt}\varepsilon_{2\pt}} \rangle.$$
 			
 			\item If $x\pm1$ is \textbf{not} a square in $\NN$, then $\{\varepsilon_{2},\varepsilon_{\po\pt}, \sqrt{\varepsilon_{\po\pt}\varepsilon_{2\po\pt}}\}$ is the F.S.U of $k_3$. Thus,    
 			 	$$E_{k_1}E_{k_2}E_{k_3}=\langle-1,  \varepsilon_{2}, \varepsilon_{\po},\varepsilon_{\pt},\  \varepsilon_{\po\pt},\sqrt{\varepsilon_{2}\varepsilon_{\po}\varepsilon_{2\po}}, \sqrt{\varepsilon_{2}\varepsilon_{\pt}\varepsilon_{2\pt}}, ,\sqrt{\varepsilon_{\po\pt}\varepsilon_{2\po\pt}}   \rangle.$$ 
 			 	As  $\varepsilon_{\po\pt}$ is a square in $\KK$ $($cf. Remark \ref{twosquare}$)$, then let us consider $\chi \in \KK  $ such that: 
 			\begin{eqnarray}
 				\chi^2=\varepsilon_{2}^a\varepsilon_{\po}^b \varepsilon_{\pt}^c\sqrt{\varepsilon_{2}\varepsilon_{\po}\varepsilon_{2\po}}^d\sqrt{\varepsilon_{2}\varepsilon_{\pt}\varepsilon_{2\pt}}^e\sqrt{\varepsilon_{\po\pt}\varepsilon_{2\po\pt}}^f,	 
 			\end{eqnarray}
 			where $a, b, c, d, e$ and $f$ are in $\{0, 1\}$. Then by applying the six norm maps we get 
 			\begin{eqnarray}
 				N_{\KK/k_2}(\chi^2)&=&
 				\varepsilon_2^{2a}\cdot (-1)^b \cdot\varepsilon_{\pt}^{2c}\cdot (-1)^{a_1 d}\cdot\varepsilon_{2}^d \cdot (\varepsilon_{2}\varepsilon_{\pt}\varepsilon_{2\pt})^{ e} \cdot   (-1)^{b_1 f},\label{firstt}\\
 				N_{\KK/k_3}(\chi^2)&=&
 				\varepsilon_2^{2a}\cdot (-1)^b \cdot(-1)^c \cdot \varepsilon_{2}^e \cdot (-1)^{a_2 e} \cdot (\varepsilon_{\po\pt}\varepsilon_{2\po\pt})^f, \label{secondd}\\
 				N_{\KK/k_4}(\chi^2)&=&
 				(-1)^a\cdot \varepsilon_{\po}^{2b} \cdot\varepsilon_{\pt}^{2c} \cdot  \varepsilon_{  \po\pt}^{f} \cdot  (-1)^{a_3 f},\nonumber\\
 				N_{\KK/k_6}(\chi^2)&=&
 				(-1)^a\cdot \varepsilon_{\po}^{2b} \cdot(-1)^{c}\cdot \varepsilon_{  2\po\pt}^{f} \cdot(-1)^{a_4 f},\label{fivee}\\
 				N_{\KK/k_1}(\chi^2)&=&
 				\varepsilon_{2}^{2a}\cdot \varepsilon_{\po}^{2b} \cdot(-1)^{c} \cdot(-1)^{a_5 f},\nonumber\\
 				N_{\KK/k_5}(\chi^2)&=& (-1)^{a}\cdot(-1)^{b}\cdot\varepsilon_{\pt}^{2c}\cdot \varepsilon_{2\po\pt}^{f}
 				\cdot(-1)^{a_6 f }, \label{sixx}                
 			\end{eqnarray}
 		where the $a_i$ and $b_1$ are in $\{0,1\}$. From the first equality \eqref{firstt}, we deduce that $d=0$ because $\varepsilon_{2}$	is not a square in $k_2$. The same goes for $e=0$ from the second equality \eqref{secondd}, and as in proof of   $7)$, we deduce from \eqref{fivee}, \eqref{sixx} and the that fact that  
 		either  $\sqrt{\varepsilon_{  2\po\pt}}\in  k_5$ or  $\sqrt{\varepsilon_{  2\po\pt}}\in k_6  $ that we have $f=0$. By putting together the congruence equations we deduce that $a=b=c=d=e=f=0$. Thus, all the forms of $\chi^2$ are eliminated. Hence, as above, we get
 		 	$$E_{\KK} =\langle-1,  \varepsilon_{2}, \varepsilon_{\po},\varepsilon_{\pt},\  \sqrt{\varepsilon_{\po\pt}},\sqrt{\varepsilon_{2}\varepsilon_{\po}\varepsilon_{2\po}}, \sqrt{\varepsilon_{2}\varepsilon_{\pt}\varepsilon_{2\pt}}, \sqrt{ \varepsilon_{2\po\pt}}   \rangle.$$

 		\end{enumerate}

 			\item Finally assume that  $(n_1,n_2,n_3, n_4)=( 1, 1, 1,  1)$. By Lemmas \ref{fork1k2} and \ref{lmunit}, we have:  
 			\begin{enumerate}[$\star$]
 				\item A F.S.U of $k_i$ is given by $\{\varepsilon_{2},\varepsilon_{p_i}, \sqrt{\varepsilon_{2p_i}}\}$ for $i\in \{1,2\}$, 
 				\item A F.S.U of $k_3$ is given by $\{\varepsilon_{2},\varepsilon_{\po\pt}, \sqrt{\varepsilon_{2\po\pt}}\}$ or $\{\varepsilon_{2},\varepsilon_{\po\pt},\sqrt{\varepsilon_{\po\pt}\varepsilon_{2\po\pt}}\}$ according to whether $x\pm1$ is a square in $\NN$ or not.
 			\end{enumerate} 
 			\begin{enumerate}[$\bullet$]
 				\item If $x\pm1$ is a square in $\NN$, we have  
 			 	$$E_{k_1}E_{k_2}E_{k_3}=\langle-1,  \varepsilon_{2}, \varepsilon_{\po},\varepsilon_{\pt},\varepsilon_{\po\pt},\sqrt{\varepsilon_{2\po}}, \sqrt{\varepsilon_{2\pt}} ,\sqrt{\varepsilon_{2\po\pt}}   \rangle.$$ 
 			 	By Remark \ref{twosquare}, $\varepsilon_{\po\pt}$ is a square in $\KK$. So let us consider $\chi \in \KK  $ such that: 
 				\begin{eqnarray}
 					\chi^2=\varepsilon_{2}^a\varepsilon_{\po}^b \varepsilon_{\pt}^c\sqrt{\varepsilon_{2\po}}^d\sqrt{\varepsilon_{2\pt}}^e\sqrt{\varepsilon_{2\po\pt}}^f,	 
 				\end{eqnarray}
 				where $a, b, c, d, e$ and $f$ are in $\{0, 1\}$.   We shall proceed as above, by using the
 				norm maps, and Lemma \ref{calcul} to compute the norms of $\sqrt{\varepsilon_{2\po}}$ and $\sqrt{\varepsilon_{2\pt}}$ which we will denote by $u$ and $v$, respectively, for their parameters.
 			 
 				\noindent\ding{229}   Let us start	by applying   the norm map $N_{\KK/k_2}=1+\tau_2$. We have
 				\begin{eqnarray*}
 					N_{\KK/k_2}(\chi^2)&=&
 					\varepsilon_{2}^{2a}\cdot(-1)^b \cdot \varepsilon_{\pt}^{2c}\cdot (-1)^{du} \cdot \varepsilon_{2\pt}^{e} \cdot (-1)^{a_1 f} \geq 0
 				\end{eqnarray*}
 				for some $a_1 \in \{0,1\}$.	Thus, $b+du+a_1 f\equiv 0\pmod2$.

 				\noindent\ding{229}   By applying   the norm map $N_{\KK/k_3}=1+\tau_2\tau_3$. We have
 				\begin{eqnarray*}
 					N_{\KK/k_3}(\chi^2)&=&
 					\varepsilon_{2}^{2a}\cdot(-1)^b \cdot(-1)^c  \cdot (-1)^{du} \cdot (-1)^{ve}\cdot \varepsilon_{2\po\pt}^{f}
 				\end{eqnarray*}
 				thus $b+c+du+ve\equiv 0\pmod2$.
 				
 				\noindent\ding{229}   By applying   the norm map $N_{\KK/k_4}=1+\tau_1$. We have
 				\begin{eqnarray*}
 					N_{\KK/k_4}(\chi^2)&=&
 					(-1)^a\cdot \varepsilon_{\po}^{2b}\cdot\varepsilon_{\pt}^{2c} \cdot(-1)^{d(u+1)}\cdot (-1)^{e(v+1)} \cdot(-1)^{a_2 f} 
 				\end{eqnarray*}
 				for some $a_2 \in \{0,1\}$.	Thus, $a+d(u+1)+e(v+1)+a_2 f\equiv 0\pmod2$.

 				\noindent\ding{229}   By applying   the norm map $N_{\KK/k_6}=1+\tau_1\tau_3$. We have
 				\begin{eqnarray*}
 					N_{\KK/k_6}(\chi^2)&=&
 					(-1)^a\cdot \varepsilon_{\po}^{2b}\cdot(-1)^c \cdot(-1)^{d(u+1)}\cdot (-1)^{e} \cdot \varepsilon_{2\pt}^{e}\cdot\varepsilon_{2\po\pt}^f (-1)^{a_3 f}
 				\end{eqnarray*}
 				for some $a_3 \in \{0,1\}$.	Thus, $a+c+d(u+1)+e+a_3 f\equiv 0\pmod2$. As $\varepsilon_{2\pt}$ and $\varepsilon_{2\po\pt}$ are not squares in $k_6$ then $e=f.$
 				
 				\noindent\ding{229}   By applying   the norm map $N_{\KK/k_5}=1+\tau_1\tau_2$. We have
 				\begin{eqnarray*}
 					N_{\KK/k_5}(\chi^2)&=&
 					(-1)^a\cdot (-1)^b \cdot \varepsilon_{\pt}^{2c}\cdot(-1)^d\cdot\varepsilon_{2\po}^{d} \cdot(-1)^{e(v+1)} \cdot \varepsilon_{2\po\pt}^f (-1)^{a_4 f} 
 				\end{eqnarray*}
 				for some $a_4 \in \{0,1\}$.	Thus, $a+b+d+e(v+1)+a_4 f\equiv 0\pmod2$. As $\varepsilon_{2\po}$ and $\varepsilon_{2\po\pt}$ are not squares in $k_6$ then $d=f.$ Hence
 			 	\begin{eqnarray}
 					\chi^2=\varepsilon_{2}^a\varepsilon_{\po}^b \varepsilon_{\pt}^c\sqrt{\varepsilon_{2\po}\varepsilon_{2\pt}\varepsilon_{2\po\pt}}^d,	 
 				\end{eqnarray}
 			 	By applying the other norms, we deduce nothing. So we have   our result for this case.

 				\item If $x\pm1$ is not a square in $\NN$, then 
 				$$E_{k_1}E_{k_2}E_{k_3}=\langle-1,  \varepsilon_{2}, \varepsilon_{\po},\varepsilon_{\pt},\varepsilon_{\po\pt},\sqrt{\varepsilon_{2\po}}, \sqrt{\varepsilon_{2\pt}} ,{\sqrt{\varepsilon_{\po\pt}\varepsilon_{2\po\pt}}}   \rangle.$$   
 				As by Remark \ref{twosquare}, $\varepsilon_{\po\pt}$ is a square in $\KK$, we    consider $\chi \in \KK  $ such that:
 				\begin{eqnarray}
 					\chi^2=\varepsilon_{2}^a\varepsilon_{\po}^b \varepsilon_{\pt}^c\sqrt{\varepsilon_{2\po}}^d\sqrt{\varepsilon_{2\pt}}^e\sqrt{\varepsilon_{2\po\pt}\varepsilon_{\po\pt}}^f,	 
 				\end{eqnarray}
 				where $a, b, c, d, e$ and $f$ are in $\{0, 1\}$.   Let $u$ (resp. $v$) be the parameter in  
 				Lemma \ref{calcul} for $\sqrt{\varepsilon_{2\po}}$,  i.e. $p=p_1$  (resp. $\sqrt{\varepsilon_{2\pt}}$, i.e. $p=p_2$).\\
 				\noindent\ding{229}   Let us start	by applying   the norm map $N_{\KK/k_2}=1+\tau_2$. We have
 				\begin{eqnarray*}
 					N_{\KK/k_2}(\chi^2)&=&
 					\varepsilon_{2}^{2a}\cdot(-1)^b \cdot \varepsilon_{\pt}^{2c}\cdot (-1)^{du} \cdot \varepsilon_{2\pt}^{e} \cdot (-1)^{a_1 f} 
 				\end{eqnarray*}
 				for some $a_1 \in \{0,1\}$.	Thus, $b+du+a_1 f\equiv 0\pmod2$.\\
 				\noindent\ding{229}   By applying   the norm map $N_{\KK/k_3}=1+\tau_2\tau_3$, we get
 				\begin{eqnarray*}
 					N_{\KK/k_3}(\chi^2)&=&
 					\varepsilon_{2}^{2a}\cdot(-1)^b \cdot(-1)^c  \cdot (-1)^{du} \cdot (-1)^{ve}\cdot (\varepsilon_{\po\pt}\varepsilon_{2\po\pt})^{f}
 				\end{eqnarray*}
 				thus $b+c+du+ve\equiv 0\pmod2$.\\
 				\noindent\ding{229}   By applying   the norm map $N_{\KK/k_4}=1+\tau_1$, we get
 				\begin{eqnarray*}
 					N_{\KK/k_4}(\chi^2)&=&
 					(-1)^a\cdot \varepsilon_{\po}^{2b}\cdot\varepsilon_{\pt}^{2c} \cdot(-1)^{d(u+1)}\cdot (-1)^{e(v+1)}\cdot \varepsilon_{\po\pt}^f \cdot(-1)^{a_2 f}
 				\end{eqnarray*}
 				for some $a_2 \in \{0,1\}$.	Thus, $a+d(u+1)+e(v+1)+a_2 f\equiv 0\pmod2$.\\
 				\noindent\ding{229}   By applying   the norm map $N_{\KK/k_6}=1+\tau_1\tau_3$, we get
 				\begin{eqnarray}\label{eqqqq32}
 					N_{\KK/k_6}(\chi^2)&=&
 					(-1)^a\cdot \varepsilon_{\po}^{2b}\cdot(-1)^c \cdot(-1)^{d(u+1)}\cdot (-1)^{e} \cdot \varepsilon_{2\pt}^{e}\cdot\varepsilon_{2\po\pt}^f (-1)^{a_3 f} 
 				\end{eqnarray}
 				for some $a_3 \in \{0,1\}$.	Thus, $a+c+d(u+1)+e+a_3 f\equiv 0\pmod2$. \\
 				\noindent\ding{229}   Let now apply  the norm map $N_{\KK/k_5}=1+\tau_1\tau_2$. We have
 				\begin{eqnarray}
 					N_{\KK/k_5}(\chi^2)&=&
 					(-1)^a\cdot (-1)^b \cdot \varepsilon_{\pt}^{2c}\cdot(-1)^d\cdot\varepsilon_{2\po}^{d} \cdot(-1)^{e(v+1)} \cdot \varepsilon_{2\po\pt}^f (-1)^{a_4 f}\label{eqqqq34} 
 				\end{eqnarray}
 				for some $a_4 \in \{0,1\}$.	Thus, $a+b+d+e(v+1)+a_4 f\equiv 0\pmod2$\\
 				\noindent\ding{229}   Finally by applying   the norm map $N_{\KK/k_1}=1+\tau_3$, we get
 				\begin{eqnarray*}
 					N_{\KK/k_1}(\chi^2)&=&\varepsilon_{2}^{2a}\cdot\varepsilon_{\po}^{2b}\cdot
 					(-1)^c\cdot\varepsilon_{2\po}^{d} (-1)^{ev} \cdot (-1)^{a_5 f}  
 				\end{eqnarray*}
 				for some $a_5 \in \{0,1\}$.	Thus, $c+ev+a_5 d\equiv 0\pmod2$.\\
 				So we deduce from equalities \eqref{eqqqq32} and \eqref{eqqqq34} that we have ($e=f$ and $d=0$) or ($d=f$ and $e=0$) according to whether $\sqrt{\varepsilon_{2\po\pt}}\in k_5$ or $\sqrt{\varepsilon_{2\po\pt}}\in k_6$ following the second or the third form of  $\sqrt{\varepsilon_{2\po\pt}}$ in Remark \ref{twosquare}. Thus, 
 				$\chi^2=\varepsilon_{2}^a\varepsilon_{\po}^b \varepsilon_{\pt}^c\sqrt{\varepsilon_{2\pt}\varepsilon_{\po\pt}\varepsilon_{2\po\pt}}^e$ in the first case, and 
 				$\chi^2=\varepsilon_{2}^a\varepsilon_{\po}^b \varepsilon_{\pt}^c\sqrt{\varepsilon_{2\po}\varepsilon_{\po\pt}\varepsilon_{2\po\pt}}^d$ in the second case.
 				Hence the last item.
 				 		 		
 			\end{enumerate}
 		\end{enumerate}
 	\end{proof}

 	\section{\bf The case: $\po \equiv 1 \pmod8$, $  \pt\equiv 1 \pmod4$   and  $(n_3, n_4)=(-1,-1)$}$\,$ 		\label{section23}		
 	
 	In this section, we compute the unit group of $\KK=\QQ(\sqrt 2, \sqrt{p_1}, \sqrt{p_2} )$, where $\po \equiv 1 \pmod8$ and  $  \pt\equiv 1 \pmod4$ are two prime numbers such that $(n_3, n_4)=(-1,-1)$. Note that a possible interchanging of $p_1$ and $p_2$ is taken into account.
 	
 	\begin{theorem}\label{MT4} Let $\po \equiv 1 \pmod8$ and  $  \pt\equiv 1 \pmod4$  be two prime numbers   such that   $(n_3, n_4)=(-1,-1)$. Put $\KK=\QQ(\sqrt 2, \sqrt{\po}, \sqrt{\pt} )$.
 		\begin{enumerate}[\rm $1)$]
 			\item   Assume that $(n_1,n_2,n_3, n_4)=(-1,-1,-1,-1)$. Then the unit group of $\KK$ takes one of the following forms:
 			
 			\begin{enumerate}[\rm $a)$]
 				\item  
 				$E_{\KK}=\langle -1,  \varepsilon_{2}, \varepsilon_{\po},   \varepsilon_{\pt}, \sqrt{\varepsilon_{2}\varepsilon_{\po}\varepsilon_{2\po}}, 
 				\sqrt{\varepsilon_{2}\varepsilon_{\pt} \varepsilon_{2\pt}}, \sqrt{\varepsilon_{2}\varepsilon_{\po\pt}\varepsilon_{2\po\pt}},\sqrt{\varepsilon_{\po}\varepsilon_{\pt}\varepsilon_{\po\pt}}  \rangle,$ 
 				
 				\item  
 				$E_{\KK}=\langle -1,  \varepsilon_{2},   \varepsilon_{\po},\varepsilon_{\pt}, \sqrt{\varepsilon_{2}\varepsilon_{\po}\varepsilon_{2\po}},\sqrt{\varepsilon_{2}\varepsilon_{\pt}\varepsilon_{2\pt}}, \sqrt{\varepsilon_{\po}\varepsilon_{\pt}\varepsilon_{\po\pt}},\\
 				\qquad \sqrt{\varepsilon_{2}^a\varepsilon_{\po}^b \varepsilon_{\pt}^c \varepsilon_{\po\pt}^d \sqrt{\varepsilon_{2}\varepsilon_{\po}\varepsilon_{2\po}}\sqrt{\varepsilon_{2}\varepsilon_{\pt}\varepsilon_{2\pt}}
 					\sqrt{\varepsilon_{2}\varepsilon_{\po\pt}\varepsilon_{2\po\pt} }}   \rangle,$
 				
 				for some  $a$, $b$, $c$ and $d$ in $\{0,1\}$ with $(b,c,d)\not=(1,1,1)$.
 				
 			\end{enumerate}
 			{\bf Furthermore},  if  the unit group of $\KK$ takes the form $b)$, then $\varepsilon_{2}\varepsilon_{\po\pt}\varepsilon_{2\po\pt}$ is  a square in $k_3$.

 			\item   Assume that $(n_1,n_2,n_3, n_4)=(-1,1,-1,-1)$. Then the unit group of $\KK$ is :
 			$$E_{\KK}=\langle-1,  \varepsilon_{2}, \varepsilon_{\po},\varepsilon_{\pt} ,\sqrt{\varepsilon_{2\pt}}, \sqrt{\varepsilon_{2}\varepsilon_{\po}\varepsilon_{2\po}},\sqrt{\varepsilon_2\varepsilon_{\po\pt}\varepsilon_{2\po\pt}}, \sqrt{\varepsilon_{\po}\varepsilon_{\pt}\varepsilon_{\po\pt}} \rangle.$$

 			\item   Assume that $(n_1,n_2,n_3, n_4)=(1,1,-1,-1)$. Then the unit group of $\KK$ is :
 			$$E_{\KK}=\langle-1,  \varepsilon_{2}, \varepsilon_{\po},\varepsilon_{\pt} ,\sqrt{\varepsilon_{2\po}},\sqrt{\varepsilon_{2\pt}}, \sqrt{\varepsilon_{2}\varepsilon_{\po\pt}\varepsilon_{2\po\pt}}, \sqrt{\varepsilon_{\po}\varepsilon_{\pt}\varepsilon_{\po\pt}} \rangle.$$

 		\end{enumerate}
 	\end{theorem}
 	\begin{proof}$\;$
 		\begin{enumerate}[\rm $1)$]
 			\item  Assume that $(n_1,n_2,n_3, n_4)=(-1,-1,-1,-1)$. We have   $N(\varepsilon_{\po})=N(\varepsilon_{\pt})=n_1=n_2=n_3= n_4=-1$. Thus,    for $i\in\{1,2\}$,  
 			a F.S.U of $k_i$ is given by $\{\varepsilon_{2}, \varepsilon_{p_i},	\sqrt{\varepsilon_{2}\varepsilon_{p_i}\varepsilon_{2p_i}}\}$ (cf. Lemma \ref{fork1k2}). Furthermore, by Lemma 
 			\ref{lmunit},  a F.S.U of $k_3$ is : 
 			\begin{enumerate}[$\star$]
 				\item   $\{\varepsilon_{2}, \varepsilon_{\po\pt},	\sqrt{\varepsilon_{2}\varepsilon_{\po\pt}\varepsilon_{2\po\pt}}\}$, if $\varepsilon_{2}\varepsilon_{\po\pt}\varepsilon_{2\po\pt}$ is a square in $k_3$,
 				\item   $\{\varepsilon_{2}, \varepsilon_{\po\pt},\varepsilon_{2\po\pt}\}$ else.
 			\end{enumerate}
 			\begin{enumerate}[\rm $a)$]
 				\item If $\varepsilon_{2}\varepsilon_{\po\pt}\varepsilon_{2\po\pt}$ is \textbf{not} a square in $k_3$,   then  $$E_{k_1}E_{k_2}E_{k_3}=\langle-1,  \varepsilon_{2}, \varepsilon_{\po},\varepsilon_{\pt}, \varepsilon_{\po\pt} ,\varepsilon_{2\po\pt},  \sqrt{\varepsilon_{2}\varepsilon_{\po}\varepsilon_{2\po}},\sqrt{\varepsilon_{2}\varepsilon_{\pt}\varepsilon_{2\pt}} \rangle.$$
 				So we shall consider 
 				\begin{eqnarray}\label{chi1t4}
 					\chi^2=\varepsilon_{2}^a\varepsilon_{\po}^b \varepsilon_{\pt}^c\varepsilon_{\po\pt}^d \varepsilon_{2\po\pt}^e\sqrt{\varepsilon_{2}\varepsilon_{\po}\varepsilon_{2\po}}^f\sqrt{\varepsilon_{2}\varepsilon_{\pt}\varepsilon_{2\pt}}^g,	 
 				\end{eqnarray} 
 				where $a, b, c, d, e, f$ and $g$ are in $\{0, 1\}$. Notice that the same proof of the second item of Theorem \ref{MT1A} gives the result.

 				\item If $\varepsilon_{2}\varepsilon_{\po\pt}\varepsilon_{2\po\pt}$ is  a square in $k_3$,   then the proof is the same as in the third item of Theorem \ref{MT1A}.			\end{enumerate}

 			\item Assume that $(n_1,n_2,n_3, n_4)=(-1,1,-1,-1)$. By Lemma \ref{fork1k2},  FSUs of $k_1$ and  $k_2$ and   are given by 
 			$\{\varepsilon_{2}, \varepsilon_{\po},	\sqrt{\varepsilon_{2}\varepsilon_{\po}\varepsilon_{2\po}}\}$ and $\{\varepsilon_{2}, \varepsilon_{\pt},	\sqrt{\varepsilon_{2\pt}}\}$ respectively.
 			\begin{enumerate}[$\bullet$]
 				\item If $\varepsilon_2\varepsilon_{\po\pt}\varepsilon_{2\po\pt}$ is a square in $k_3$, then $\{\varepsilon_{2}, \varepsilon_{\po\pt},	\sqrt{\varepsilon_{2}\varepsilon_{\po\pt}\varepsilon_{2\po\pt}}\}$  is a FSU of $k_3$. Thus,   $$E_{k_1}E_{k_2}E_{k_3}=\langle-1,  \varepsilon_{2}, \varepsilon_{\po},\varepsilon_{\pt}, \varepsilon_{\po\pt} ,\sqrt{\varepsilon_{2\pt}}, \sqrt{\varepsilon_{2}\varepsilon_{\po}\varepsilon_{2\po}},\sqrt{\varepsilon_2\varepsilon_{\po\pt}\varepsilon_{2\po\pt}} \rangle.$$
 				So let $\chi\in \KK$ be such that
 				\begin{eqnarray}\label{chi19a}
 					\chi^2=\varepsilon_{2}^a\varepsilon_{\po}^b \varepsilon_{\pt}^c\varepsilon_{\po\pt}^d \sqrt{\varepsilon_{2\pt}}^e\sqrt{\varepsilon_{2}\varepsilon_{\po}\varepsilon_{2\po}}^f\sqrt{\varepsilon_{2}\varepsilon_{\po\pt}\varepsilon_{2\po\pt}}^g,	 
 				\end{eqnarray} 
 				where $a, b, c, d, e, f$ and $g$ are in $\{0, 1\}$.   
 				Applying  $N_{\KK/k_2}, N_{\KK/k_3}, N_{\KK/k_4}$ and $N_{\KK/k_6}$ respectively, we check that 
 				\begin{equation*}
 					a=f=g=e=0\;\; \text{and} \;\; b=c=d.
 				\end{equation*}
 				Thus, we eliminated all the forms of $\chi^2$ that except the following
 				\begin{eqnarray}
 					\chi^2=(\varepsilon_{\po} \varepsilon_{\pt}\varepsilon_{\po\pt})^b.	 
 				\end{eqnarray} 
 				Hence, by Lemma \ref{squaretest}, we have 
 			 	$$E_{\KK}=\langle-1,  \varepsilon_{2}, \varepsilon_{\po},\varepsilon_{\pt} ,\sqrt{\varepsilon_{2\pt}}, \sqrt{\varepsilon_{2}\varepsilon_{\po}\varepsilon_{2\po}},\sqrt{\varepsilon_2\varepsilon_{\po\pt}\varepsilon_{2\po\pt}}, \sqrt{\varepsilon_{\po}\varepsilon_{\pt}\varepsilon_{\po\pt}} \rangle.$$

 				\item If $\varepsilon_2\varepsilon_{\po\pt}\varepsilon_{2\po\pt}$ is \textbf{not} a square in $k_3$, then $\{\varepsilon_{2}, \varepsilon_{\po\pt},	\varepsilon_{2\po\pt}\}$, is a FSU of $k_3$. Thus,  $$E_{k_1}E_{k_2}E_{k_3}=\langle-1,  \varepsilon_{2}, \varepsilon_{\po},\varepsilon_{\pt}, \varepsilon_{\po\pt}, \varepsilon_{2\po\pt} ,\sqrt{\varepsilon_{2\pt}}, \sqrt{\varepsilon_{2}\varepsilon_{\po}\varepsilon_{2\po}} \rangle.$$
 				Let us take $\chi$ an element of $\KK$ which is the square root of an element of $E_{k_1}E_{k_2}E_{k_3}$. Therefore, we can assume
 				that 
 				\begin{eqnarray}\label{chi19}
 					\chi^2=\varepsilon_{2}^a\varepsilon_{\po}^b \varepsilon_{\pt}^c\varepsilon_{\po\pt}^d\varepsilon_{2\po\pt}^e \sqrt{\varepsilon_{2\pt}}^f\sqrt{\varepsilon_{2}\varepsilon_{\po}\varepsilon_{2\po}}^g,	 
 				\end{eqnarray} 
 				where $a, b, c, d, e, f$ and $g$ are in $\{0, 1\}$.  
 				Then using norm maps we get: 
 				\begin{eqnarray}
 					N_{\KK/k_2}(\chi^2)&=& \varepsilon_{2}^{2a}\cdot (-1)^b\cdot \varepsilon_{\pt}^{2c}\cdot (-1)^d\cdot(-1)^e\cdot\varepsilon_{2\pt}^f\cdot\varepsilon_{2}^{g}\cdot (-1)^{\alpha g}, \\
 						N_{\KK/k_3}(\chi^2)&=& \varepsilon_{2}^{2a}\cdot (-1)^b\cdot (-1)^c\cdot \varepsilon_{\po\pt}^{2d}\cdot\varepsilon_{2\po\pt}^{2e}\cdot(-1)^{uf}, \\
 				N_{\KK/k_4}(\chi^2)&=& (-1)^a\cdot \varepsilon_{\po}^{2b}\cdot \varepsilon_{\pt}^{2c}\cdot\varepsilon_{\po\pt}^{2d} \cdot(-1)^e\cdot  (-1)^{(u+1)f},  \\
 						N_{\KK/k_6}(\chi^2)&=& (-1)^a\cdot \varepsilon_{\po}^{2b}\cdot (-1)^c\cdot(-1)^d \cdot \varepsilon_{2\po\pt}^e \cdot(-1)^{uf}\varepsilon_{2\pt}^f,  \\
 					N_{\KK/k_1}(\chi^2)&=& \varepsilon_{2}^{2a}\cdot \varepsilon_{\po}^{2b}\cdot (-1)^c\cdot(-1)^d \cdot (-1)^e . 			
 				\end{eqnarray}
 				With $\alpha \in \{0,1\}$ and $u \in \{0,1\}$ as in Lemma \ref{calcul}. As $\varepsilon_{2}$ and $\varepsilon_{2\pt}$  are not a squares in $k_2$ and $k_6$ respectively then we conclude that $f=g=0$, and then we have the system :
 				\[
 				\left \{
 				\begin{array}{ccc}
 					b+d+ e & \equiv& 0 \pmod 2, \\
 					b+c& \equiv& 0 \pmod 2, \\
 					a+e & \equiv& 0 \pmod 2,  
 				 	\end{array}
 				\right.
 				\]
 		  	which provides the following possible solutions based on the above equalities:   $(a,b,c,d,e)\in \{(0,1,1,1,0),(1,1,1,0,1),(1,0,0,1,1)\}$  $\pmod 2$. Therefore, the only possible forms of $\chi$ are:
 				\begin{equation*}
 					\sqrt{\varepsilon_{\po}\varepsilon_{\pt}\varepsilon_{\po\pt}}, \sqrt{\varepsilon_{2}\varepsilon_{\po}\varepsilon_{\pt}\varepsilon_{2\po\pt}}, \text{ and } \sqrt{\varepsilon_{2}\varepsilon_{\po\pt}\varepsilon_{2\po\pt}}.
 				\end{equation*}
 				Noting  that we can build the second from the two others and $\varepsilon_{\po\pt}^{-1}$, we have
 				$$E_{\KK}=\langle-1,  \varepsilon_{2}, \varepsilon_{\po},\varepsilon_{\pt} ,\sqrt{\varepsilon_{2\pt}}, \sqrt{\varepsilon_{2}\varepsilon_{\po}\varepsilon_{2\po}},\sqrt{\varepsilon_2\varepsilon_{\po\pt}\varepsilon_{2\po\pt}}, \sqrt{\varepsilon_{\po}\varepsilon_{\pt}\varepsilon_{\po\pt}} \rangle.$$

 			\end{enumerate}
 				
 				\item Assume that $(n_1,n_2,n_3, n_4)=(1,1,-1,-1)$. By Lemma \ref{fork1k2}, a F.S.U of $k_i$ is given by $\{\varepsilon_{2}, \varepsilon_{p_i},	\sqrt{\varepsilon_{2p_i}}\}$ for $i=1,2$,  and for $k_3$ we have two cases : 
 				            \begin{enumerate}[$\bullet$]
 					\item If $\varepsilon_2\varepsilon_{\po\pt}\varepsilon_{2\po\pt}$ is a square in $k_3$, then $\{\varepsilon_{2}, \varepsilon_{\po\pt},	\sqrt{\varepsilon_{2}\varepsilon_{\po\pt}\varepsilon_{2\po\pt}}\}$, is a FSU of $k_3$. Thus,  $$E_{k_1}E_{k_2}E_{k_3}=\langle-1,  \varepsilon_{2}, \varepsilon_{\po},\varepsilon_{\pt}, \varepsilon_{\po\pt} ,\sqrt{\varepsilon_{2\po}}, \sqrt{\varepsilon_{2\pt}},\sqrt{\varepsilon_2\varepsilon_{\po\pt}\varepsilon_{2\po\pt}} \rangle.$$
 					So we consider $\chi\in\KK$ such that
 					\begin{eqnarray} 
 						\chi^2=\varepsilon_{2}^a\varepsilon_{\po}^b \varepsilon_{\pt}^c\varepsilon_{\po\pt}^d \sqrt{\varepsilon_{2\po}}^e\sqrt{\varepsilon_{2\pt}}^f\sqrt{\varepsilon_{2}\varepsilon_{\po\pt}\varepsilon_{2\po\pt}}^g,	 
 					\end{eqnarray} 
 					where $a, b, c, d, e, f$ and $g$ are in $\{0, 1\}$.
 				 	Applying  $N_{\KK/k_2}, N_{\KK/k_3}, N_{\KK/k_4}, N_{\KK/k_5}$ and $N_{\KK/k_6},$ one can easily check that  
 					\begin{equation*}
 						a=f=g=e=0\;\; \text{and} \;\; b=c=d
 					\end{equation*}

 Therefore, $$\chi^2=(\varepsilon_{\po} \varepsilon_{\pt}\varepsilon_{\po\pt})^b \,$$
 					Hence, 
 					$$E_{\KK}=\langle-1,  \varepsilon_{2}, \varepsilon_{\po},\varepsilon_{\pt} ,\sqrt{\varepsilon_{2\po}},\sqrt{\varepsilon_{2\pt}}, \sqrt{\varepsilon_{2}\varepsilon_{\po\pt}\varepsilon_{2\po\pt}}, \sqrt{\varepsilon_{\po}\varepsilon_{\pt}\varepsilon_{\po\pt}} \rangle.$$
 					\item If $\varepsilon_2\varepsilon_{\po\pt}\varepsilon_{2\po\pt}$ is not a square in $k_3$, then $\{\varepsilon_{2}, \varepsilon_{\po\pt},	\varepsilon_{2\po\pt}\}$, is a FSU of $k_3$. So  $$E_{k_1}E_{k_2}E_{k_3}=\langle-1,  \varepsilon_{2}, \varepsilon_{\po},\varepsilon_{\pt}, \varepsilon_{\po\pt},\varepsilon_{2\po\pt} ,\sqrt{\varepsilon_{2\po}},\sqrt{\varepsilon_{2\pt}}\rangle.$$
 					using the same calculations as above, we obtain the same result.
 				\end{enumerate}
 			\end{enumerate}
 			\end{proof}

 		
 	\section{\bf Last remarks and results}\label{lastsec}
 	
 	We close this paper with the following results and remarks. Let us start with the next corollaries that follow directly from  our main results.
 	\begin{corollary} Assume that $p_1\equiv 1\pmod 8$ and   $p_2\equiv 5\pmod 8$. Then 
 		$$ h_2(\KK)=	\frac{1}{2^{5-\delta}}      h_2(2\po) h_2(2\pt)h_2(\po\pt)  h_2(2\po\pt),$$
 		where $\delta \in \{0,1\}$. Furthermore, if  none of the below conditions holds   $($i.e. $\mathrm{C1}$, $\mathrm{C2}$ and $\mathrm{C3}$$)$,  then $\delta =0$.
 		\begin{enumerate}
 			\item[$\mathrm{C1}:$] $n_1=n_2=n_3=n_4$,
 			\item[$\mathrm{C2}:$] $n_1=n_2=-n_3= n_4=1$ and $x\pm1$ is a square in $\NN$, $x$ and  $y$ are two integers such that $\varepsilon_{2\po\pt}=x+y\sqrt{2\po\pt} $,
 			\item[$\mathrm{C3}:$] $-n_1=n_2=n_3= n_4=1$ and $x\pm1$ is {\bf not } a square in $\NN$, $x$ and  $y$ are two integers such that $\varepsilon_{2\po\pt}=x+y\sqrt{2\po\pt} $.
 		\end{enumerate}
 	\end{corollary}
 	\begin{proof}
 		This is a direct deduction from Lemmas \ref{class numbers of quadratic field} and \ref{wada's f.}, and Theorems      \ref{MT3}  and  \ref{MT4}.
 	\end{proof}

 		\begin{corollary} Assume that $p_1\equiv 5\pmod 8$ and   $p_2\equiv 5\pmod 8$. Put $ \LL= \mathbb{Q}( \sqrt{-1},\sqrt{2}, \sqrt{\po}, \sqrt{\pt})$. Then 
 		$$ h_2(\LL)=	\frac{1}{2^{5-\delta}}         h_2(\po\pt)  h_2(2\po\pt)h_2(-\po\pt)h_2(-2\po\pt),$$
 		where $\delta \in \{0,1\}$. More precisely,  $\delta=1$ if and only if $N(\varepsilon_{\po\pt})=-1$ and $\varepsilon_{2}\varepsilon_{\po\pt}\varepsilon_{2\po\pt}$ is a square in $k_3$.
 	\end{corollary}
 		\begin{proof}
 		This is a direct deduction from the equality \eqref{frem}, Lemmas \ref{class numbers of quadratic field} and \ref{wada's f.}, the fact that $h_2(-p_1)=h_2(-p_2)=h_2(-2p_1)=h_2(-2p_2)=2$ (cf. \cite[Corollary 19.6]{connor88}) and  Theorem \ref{MT1A}.
 	\end{proof}

 	\begin{corollary} Assume that $p_1\equiv 1\pmod 8$ and   $p_2\equiv 5\pmod 8$.  Put $ \LL= \mathbb{Q}( \sqrt{-1},\sqrt{2}, \sqrt{\po}, \sqrt{\pt})$. Then 
 	$q(\LL)=2^{5+\delta}$
 		where $\delta \in \{0,1\}$. Furthermore, if  none of the below conditions holds   $($i.e. $\mathrm{C1}$, $\mathrm{C2}$ and $\mathrm{C3}$$)$,  then $\delta =0$.
 		\begin{enumerate}
 			\item[$\mathrm{C1}:$] $n_1=n_2=n_3=n_4$,
 			\item[$\mathrm{C2}:$] $n_1=n_2=-n_3= n_4=1$ and $x\pm1$ is a square in $\NN$, $x$ and  $y$ are two integers such that $\varepsilon_{2\po\pt}=x+y\sqrt{2\po\pt} $,
 		 \item[$\mathrm{C3}:$] $-n_1=n_2=n_3= n_4=1$ and $x\pm1$ is {\bf not } a square in $\NN$, $x$ and  $y$ are two integers such that $\varepsilon_{2\po\pt}=x+y\sqrt{2\po\pt} $.
 		\end{enumerate}
 	\end{corollary}
 	\begin{proof}
 		This is a direct deduction from the equality \eqref{frem} and Theorems   \ref{MT3}  and  \ref{MT4}.
 	\end{proof}

\begin{remark}
Notice that if $p_1\equiv 1\pmod 8$ and   $p_2\equiv 1\pmod 4$ are such that one of conditions  $\mathrm{C1}$, $\mathrm{C2}$ and $\mathrm{C3}$ mentioned above, is satisfied then 
Theorem        \ref{MT3} $(3) (b)$, $(7) (b)$, $(9)$ and  Theorem \ref{MT4} $(1)(b)$ give two possibilities for the fundamental system of units of $\KK$ $($resp. $\LL$$)$, that is $q(\KK)=2^4$ or $2^5$ (resp. $q(\LL)=2^5$ or $2^6$). It seems to be very difficult, in these cases, to decide    which is the exact fundamental system (resp. which is the exact value of $q(\KK)$  and so that of  $q(\LL)$). We give the following numerical examples, by means of PARI/GP calculator version  2.15.5 (64bit), Feb 11 2024, which illustrate these cases.

 \begin{table}[H]
	
	{  {\centering\small \centering\renewcommand{\arraystretch}{1.6}
			\setlength{\tabcolsep}{0.5cm}
			$$\begin{tabular}{|c|c|c|c|}
				\hline
				$p_1\ (p_1\pmod 8)$ & $p_2\ (p_2\pmod 8)$ & $(n_1, n_2,n_3,n_4)$&  $q(\KK) $\\
				\hline
				$ 89 \ (1\pmod 8 )$&$73  \ (1\pmod 8 )$& $(1,1,1,1)$ & $2^5$   \\
				 	\hline
				 $193 \ (1\pmod 8 )$	 &$ 97 \ (1\pmod 8 )$&$(1,1,1,1)$ & $2^5$ \\
				 	\hline
				 	$41 \ (1\pmod 8 ) $& $13 \ (5\pmod 8 )$  &$(-1,-1,-1,-1)$ & $2^4$ \\
				 	\hline
				 	$ 41 \ (1\pmod 8 )$&$29 \ (5\pmod 8 )$  &$(-1,-1,-1,-1)$ & $2^5$ \\
				 	\hline
				 	$ 457 \ (1\pmod 8 )$& $41 \ (1\pmod 8 )$ &$(-1,-1,-1,-1)$ & $2^4$ \\
				 	\hline
				$ 457 \ 	(1\pmod 8 )$& $113 \ (1\pmod 8 )$ &$(-1,-1,-1,-1)$ & $2^5$ \\
				 	\hline
				 $97 \ 	(1\pmod 8 )$&$ 17 \ (1\pmod 8 )$ &$(1,1,-1,1)$ & $2^5$ \\
				 	\hline
				 $281 \ 	(1\pmod 8 )$&17  \ $(1\pmod 8 ) $ &$(1,1,-1,1)$ & $2^4$ \\
				 	\hline
				 	 $41 \ 	(1\pmod 8 )$&73 \  $(1\pmod 8 ) $ &$(-1,1,1,1)$ & $2^5$ \\
				 	\hline
				 	 $41 \ 	(1\pmod 8 )$&337  \ $(1\pmod 8 ) $ &$(-1,1,1,1)$ & $2^4$ \\
				 	\hline
			\end{tabular} $$ }}
	
	\caption{Numerical examples illustrating the  cases  of $\mathrm{C1}$,  $\mathrm{C2}$ and $\mathrm{C3}$. }	 
\end{table}

 \end{remark}

 We close the paper with the following.

 	\begin{proposition}	Let $p_1\equiv 1\pmod 8$ and   $p_2\equiv 5\pmod 8$ be such that $\genfrac(){}{0}{p_1}{p_2}=-1$  and    $(n_1,n_2,n_3, n_4)=(-1,-1,-1,-1)$.
 		Then 
 		\begin{enumerate}[$a)$]
 			\item $q(k_3)=q(k_6)=2$,  that is $\sqrt{\varepsilon_{2}\varepsilon_{\po\pt}\varepsilon_{2\po\pt} }\in k_3$ and 
 			$\sqrt{  \varepsilon_{\po}\varepsilon_{2\pt}\varepsilon_{2\po\pt}}\in k_6$. 
 		  	\item $q(\KK)=2^{4}q(k_5) $ and so $q(\KK)=2^5$ or $q(\KK)=2^4$. More precisely,
 			$$q(\KK)=2^5 \text{ if and only if }    \sqrt{\varepsilon_{\pt}\varepsilon_{2\po}\varepsilon_{2\po\pt}}\in k_5.$$
 		\end{enumerate}
 	\end{proposition}
 			
 	\begin{proof} To proof this result, we need some properties of the Hilbert $2$-class field tower of a number field $k$ such that $\mathbf{C}l(k)\simeq (2,2)$. For this reason, we refer the reader to \cite{aczboh,Ki76}.
 		  Notice that $\KK$ is a quartic unramified extension of  $k=\QQ(\sqrt{2p_1p_2})$.
 		Under the conditions of the first item,  we have $\mathbf{C}l(k)\simeq (2,2)$ (cf. \cite{kaplan76}). Therefore, $k^{(1)}=\KK$.
 		By Lemma \ref{class numbers of quadratic field}, we have  $ h_2(2p_2 )=2$,  $h_2(p_1p_2)=2$  and $h_2(2p_1 ) $ is divisible by $4$.  Thus,   
 		\begin{eqnarray*}
 			h_2(k_3)= 2q(k_3), \;  \quad h_2(k_6)=2q(k_6)  \; \text{ and } \; h_2(k_5)=q(k_5)  h_2(2\po ).
 		\end{eqnarray*}
 		Notice that $k_3$, $k_5$ and $k_6$ are the three unramified quadratic extensions of $k$ within  $k^{(1)}$. 
 		If $q(k_3)=1$ or $q(k_6)=1$, then according to \cite[Figure 2]{aczboh}, we have $ h_2(k_5)=q(k_5)  h_2(2\po )=2$, which contradicts the fact that 
 		$h_2(2p_1 ) $ is divisible by $4$. Therefore,  $q(k_3) =q(k_6)=2$ and so $\sqrt{\varepsilon_{2}\varepsilon_{\po\pt}\varepsilon_{2\po\pt} }\in k_3$ and 
 		$\sqrt{  \varepsilon_{\po}\varepsilon_{2\pt}\varepsilon_{2\po\pt}}\in k_6$. It follows that $h_2(k_3)=   h_2(k_6)=4$.
 		Notice that  there are exactly three  primes of $\QQ(\sqrt{2})$ that ramify in $k_3$. So $2$-rank$(\mathbf{C}l(k))=3-1-e$, where $e$ is such that 
 		$2^e=|E_{\QQ(\sqrt{2})}/(E_{\QQ(\sqrt{2})}\cap N_{k_3/\QQ(\sqrt{2})}(k_3))|$ (cf. Lemma \ref{ambiguous class number formula}). If $\mathfrak p_{i}$ is a prime ideal of $\QQ(\sqrt{2})$ above $p_i$, then 
 		$\left(\frac{\varepsilon_2,\,p_1p_2}{ \mathfrak p_1}\right)= \left(\frac{\varepsilon_2,\,p_1}{ \mathfrak p_1}\right)= \left(\frac{2}{  p_1}\right)_4\left(\frac{  p_1}{  2}\right)_4=1$ (the last equality follows from Lemma \ref{norms} and the assumption that $n_1=-1$)	and 
 		$\left(\frac{\varepsilon_2,\,p_1p_2}{ \mathfrak p_2}\right)=\left(\frac{\varepsilon_2,\, p_2}{ \mathfrak p_2}\right)= \left(\frac{-1,\, p_2}{ p_2}\right)=1 $. Thus $e=0$ and so 
 		$2$-rank$(\mathbf{C}l(k_3))=2$. We similarly use the extension $k_6/\QQ(\sqrt{p_1})$ to check that $2$-rank$(\mathbf{C}l(k_6))=2$. Hence $\mathbf{C}l(k_3)\simeq \mathbf{C}l(k_6)\simeq (2,2)$.  
 		Subsequently the $2$-class group of $k_5$ is cyclic and we have $h_2(k_5)=	2h_2(\KK)$. Since $	h_2(\KK)= \frac{1}{2^{5}}q(\KK)   h_2(2\po)$, this implies that 
 		$$q(\KK)=2^{4}q(k_5).$$
 		which completes the proof.


 	\end{proof}
 	The next corollary follows from the above proof.
\begin{corollary}
Keep the same hypothesis of the above proposition, then $\mathbf{C}l(k_3)\simeq \mathbf{C}l(k_6)\simeq (2,2)$,  $\mathbf{C}l(k_5)$ and $\mathbf{C}l(\KK)$ are cyclic of order $q(k_5)  h_2(2\po )$ and $2q(k_5)  h_2(2\po )$.
\end{corollary}
 	
 	The next corollary follows from the above proposition and $\eqref{frem}$.	 
\begin{corollary}
	Keep the same hypothesis of the above proposition and let $\LL= \mathbb{Q}( \sqrt{-1},\sqrt{2}, \sqrt{\po}, \sqrt{\pt})$. Then $ q(\LL)=2^5$ or $2^6$. More precisely, 
	$$q(\LL)=2^6 \text{ if and only if }    \sqrt{\varepsilon_{\pt}\varepsilon_{2\po}\varepsilon_{2\po\pt}}\in k_5.$$
\end{corollary}

 \begin{remark}
 	Notice that under the conditions $\mathrm{C}1$, $\mathrm{C}2$ and $\mathrm{C}3$, we have $q(\KK)=2^4$ or $q(\KK)=2^5$, but we could not decide when we have each case. As shown previously, one can use class field
 	theory to deduce more precise results.

 	    The proof of the following corollary highlights another idea that can help to  get more precision concerning the result of 
 	 Theorem 3.2- 7)-b).
 	 
 \end{remark}

 		\begin{corollary}
 		 Let $x$ and $y$ (resp. $a'$ and $b'$) be two integers such that  $\varepsilon_{2\po\pt}=x+y\sqrt{2\po\pt}$ (resp. $\varepsilon_{\po\pt}=a'+b'\sqrt{\po\pt}$). 
 		 Assume that $(n_1,n_2,n_3,n_4)=(-1,1,1,1)$, $2p_1(x\pm1)$ is a square  in $\NN$ and  $2\po(a'\pm1)$ is {\bf  not } a square in $\NN$. Then $q(\KK)=2^4$ and $q(\LL)=2^5$.
 		\end{corollary}
 		
 		\begin{proof}
 			 We shall consider other biquadratic subextensions to compute the unit group of $\KK$.  By Lemmas \ref{lmunit} and \ref{biquadunutk6}, a F.S.U of  $k_3$, $k_5$ and $k_6$ is given by  $\{\varepsilon_{2},\varepsilon_{\po\pt}, \sqrt{\varepsilon_{\po\pt}\varepsilon_{2\po\pt}}\}$,  $\{\varepsilon_{\pt},\varepsilon_{2\po},\varepsilon_{2\po\pt} \}$ and $\{\varepsilon_{\po},\varepsilon_{2\pt}, \sqrt{\varepsilon_{2\po\pt}}\}$   respectively. Thus, we have
 			$E_{k_3}E_{k_5}E_{k_6}=\langle-1,  \varepsilon_{2}, \varepsilon_{\po},\varepsilon_{\pt}, \varepsilon_{2\po},\varepsilon_{2\pt},\varepsilon_{\po\pt},\sqrt{\varepsilon_{2\po\pt}}, \sqrt{\varepsilon_{\po\pt}\varepsilon_{2\po\pt}}   \rangle.$  As $\varepsilon_{ \po\pt}$ is a square in $\KK$, then we consider $\chi \in \KK  $ such that 
 			\begin{eqnarray*}
 				\chi^2=\varepsilon_{2}^a\varepsilon_{\po}^b \varepsilon_{\pt}^c\varepsilon_{2\po}^d\sqrt{\varepsilon_{2\po\pt}}^e\sqrt{\varepsilon_{\po\pt}\varepsilon_{2\po\pt}}^f,
 			\end{eqnarray*}
 			where $a, b, c, d, e$ and $f$ are in $\{0, 1\}$. We have:
 			\begin{eqnarray}
 				N_{\KK/k_3}(\chi^2)&=&
 				\varepsilon_2^{2a}\cdot (-1)^b \cdot(-1)^c\cdot (-1)^{ d}\cdot(-1)^{\alpha_0 e} \cdot \varepsilon_{\po\pt}^{f}\cdot \varepsilon_{2\po\pt}^{f},\nonumber\\
 				N_{\KK/k_5}(\chi^2)&=&
 				(-1)^a\cdot (-1)^b \cdot\varepsilon_{\pt}^{2c} \cdot \varepsilon_{2\po}^{2d}\cdot (-1)^{\beta e} \cdot  \varepsilon_{2\po\pt}^e \cdot (-1)^{\alpha_1 f},\label{rrr} \\
 				N_{\KK/k_6}(\chi^2)&=&
 				(-1)^a\cdot \varepsilon_{\po}^{2b} \cdot(-1)^{c} \cdot (-1)^{d} \cdot (-1)^{\alpha_2 f} \cdot \varepsilon_{  2\po\pt}^{f} \nonumber\\
 				N_{\KK/k_2}(\chi^2)&=& \varepsilon_{2}^{2a}
 				\cdot (-1)^{b} \cdot\varepsilon_{\pt}^{2c}\cdot (-1)^{d} \cdot(-1)^{\alpha_3 f},\nonumber\\
 				N_{\KK/k_4}(\chi^2)&=&
 				(-1)^{a}\cdot \varepsilon_{\po}^{2b}\cdot \varepsilon_{\pt}^{2c} \cdot (-1)^{d} \cdot (-1)^{\alpha_4 f} \cdot \varepsilon_{  \po\pt}^f,\label{eeee}        
 			\end{eqnarray}

 			where $\alpha_i  \in \{0,1\}$. As $\varepsilon_{  2\po\pt}$ is not  a square in $k_5$ and from  \eqref{rrr} we deduce that $e=0$. Thus, we have :

 			\[
 			\left \{
 			\begin{array}{ccc}
 				b+c+ d & \equiv& 0 \pmod 2 \\
 				a+b+\alpha_1 f & \equiv& 0 \pmod 2 \\
 				a+c +d+\alpha_2 f & \equiv& 0 \pmod 2 \\
 				b+d+\alpha_3 f & \equiv& 0 \pmod 2 \\
 				a+d+\alpha_4 f & \equiv& 0 \pmod 2 . 
 			\end{array}
 			\right.
 			\]
 			Since we have $2\po(a'\pm1)$ is {\bf not} a square in $\NN$ (i.e $\varepsilon_{  \po\pt}$ is not a square in $k_4$ according to the proof of Remark \ref{twosquare}), then \eqref{eeee}  implies that $f=0$ and  so the above system gives $a=b=d$ and $c=0$. Therefore, we eliminated all form of $\chi^2$ except the following
 			$\chi^2=(\varepsilon_{2}\varepsilon_{\po}\varepsilon_{2\po})^a$.  
 			As $\varepsilon_{2}\varepsilon_{\po}\varepsilon_{2\po} $ is a square in $\KK$, then
 			$$E_{\KK}=\langle-1,  \varepsilon_{2}, \varepsilon_{\po},\varepsilon_{\pt}, \sqrt{\varepsilon_{2\pt}},\sqrt{\varepsilon_{\po\pt}}, \sqrt{\varepsilon_{2\po\pt}}, \sqrt{\varepsilon_{2}\varepsilon_{\po}\varepsilon_{2\po}}  \rangle.$$
 			Hence $q(\KK)=2^4$ and \eqref{frem} completes the proof.
 		\end{proof}

 \section*{Acknowledgment}	
 	We would like to thank the referee for  the meticulous reading of our paper, the  detailed report and the helpful comments, which have greatly improved  our paper.

 	\end{document}